\documentclass[12pt,reqno]{amsart}

\usepackage{color}
\usepackage{amsmath, amsthm, amssymb}
\usepackage{amsfonts}
\usepackage[ansinew]{inputenc}
\usepackage[dvips]{epsfig}
\usepackage{graphicx}
\usepackage[english]{babel}
\usepackage{hyperref}

\usepackage{cite}
\usepackage{graphicx}
\usepackage{amscd}
\usepackage{color}
\usepackage{bm}
\usepackage{enumerate}

\usepackage{verbatim}
\usepackage{hyperref}
\usepackage{amstext}
\usepackage{latexsym}
\theoremstyle{plain}
\newtheorem{theorem}{Theorem}[section]

\newtheorem{proposition}[theorem]{Proposition}
\newtheorem{lemma}[theorem]{Lemma}
\newtheorem{remark}[theorem]{Remark}

\numberwithin{theorem}{section}
\numberwithin{equation}{section}

\newcommand{\average}{{\mathchoice {\kern1ex\vcenter{\hrule height.4pt
width 6pt depth0pt} \kern-9.7pt} {\kern1ex\vcenter{\hrule
height.4pt width 4.3pt depth0pt} \kern-7pt} {} {} }}

\def\R{\mathbb{R}}

\newcommand{\va }{\varphi }

\newcommand{\be}{\begin{equation}}
\newcommand{\ee}{\end{equation}}

\newcommand{\N}{\mathbb{N}}

\newcommand{\cH}{{\mathcal H}}

\newcommand{\eps}{\varepsilon}

\renewcommand{\epsilon}{\varepsilon}

\newcommand{\tauE}{\tau_{\mbox{\tiny $E$}}}

\begin{document}

\title[Curves and surfaces with constant nonlocal mean curvature]
{Curves and surfaces with constant nonlocal mean curvature:
meeting Alexandrov and Delaunay}

\author{Xavier Cabr\'e}
\address{X.C.: ICREA and Universitat Polit\`ecnica de Catalunya, Departament de Matem\`{a}ti\-ca  
Aplicada I, Diagonal 647, 08028 Barcelona, Spain}
\email{xavier.cabre@upc.edu}

\author[Mouhamed M. Fall]
{Mouhamed Moustapha Fall}
\address{M.M.F.: African Institute for Mathematical Sciences of Senegal, 
KM 2, Route de Joal, B.P. 14 18. Mbour, S\'en\'egal}
\email{mouhamed.m.fall@aims-senegal.org}

\author{Joan Sol\`a-Morales}
\address{J.S-M.: Universitat Polit\`ecnica de Catalunya, Departament de Matem\`{a}tica  Aplicada I, 
Diagonal 647, 08028 Barcelona, Spain}
\email{jc.sola-morales@upc.edu}

\author{Tobias Weth}
\address{T.W.: Goethe-Universit\"{a}t Frankfurt, Institut f\"{u}r Mathematik.
Robert-Mayer-Str. 10 D-60054 Frankfurt, Germany}
\email{weth@math.uni-frankfurt.de}

\thanks{The first and third authors are supported by grant MINECO MTM2011-27739-C04-01 and 
are part of the Catalan research group 2014 SGR 1083. The second author's work is supported by the Alexander von Humboldt foundation.}


\begin{abstract}
We are concerned with hypersurfaces of $\mathbb{R}^N$ with constant nonlocal (or fractional) mean curvature.
This is the equation associated to critical points of the fractional perimeter under a volume constraint. 
Our results are twofold. First we prove the nonlocal analogue of the Alexandrov result characterizing 
spheres as the only closed embedded hypersurfaces in $\mathbb{R}^N$ with constant mean curvature.
Here we use the moving planes method. Our second result establishes the existence of periodic bands
or ``cylinders'' in $\mathbb{R}^2$ with constant nonlocal mean curvature and bifurcating from 
a straight band. These are Delaunay type bands in the nonlocal setting. Here we use a Lyapunov-Schmidt
procedure for a quasilinear type fractional elliptic equation.
\end{abstract}

\maketitle

\section{Introduction and main results}
\label{sec:main-result}
Let $\alpha \in (0,1),$ and let $E$ be an open set in $\R^N$  (not necessarily connected, neither bounded) 
with $C^2$-boundary. Then for every $x \in \partial E$, the nonlocal or fractional mean curvature
of $\partial E$ at $x$ (that we call NMC for short) is given by
\begin{equation}
  \label{eq:def-frac-curvature}
H_E(x)= -PV \int_{\R^N} \tauE(y)|x-y|^{-(N+\alpha)}\,dy :=- \lim_{\eps \to 0} \int_{|y-x| \ge \eps} 
\tauE(y)|x-y|^{-(N+\alpha)}\,dy  
\end{equation}
and is well defined. Here and in the following, we use the notation 
$$
\tauE(y):=1_{E}(y) -1_{E^c}(y),
$$
where $E^c$ is the complement of $E$ in  $\R^N$ and $1_A$ denotes the characteristic function of $A$. 
In the first integral PV denotes the principal value sense,
and sometimes will be omitted. The minus sign in front of the integrals makes that balls have constant positive NMC.
For the asymptotics $\alpha$ tending to 0 or 1, $H_E$ should be 
renormalized with a positive constant factor $C_{N,\alpha}$.
Since constant factors are not relevant for the results of this paper, we use the simpler expression in 
(\ref{eq:def-frac-curvature}) without the constant~$C_{N,\alpha}$.

The following is a more geometric expression for the NMC. We will not use it in this paper.
We have that
\begin{equation}\label{geometric}
H_E(x)= -\frac{2}{\alpha} PV \int_{\partial E} |x-y|^{-(N+\alpha)} (x-y)\cdot\nu(y)\,dy, 
\end{equation}
where $\nu(y)$ denotes the outer unit normal to $\partial E$. This follows easily from 
\eqref{eq:def-frac-curvature} after an integration by parts using that
$\nabla_y\cdot\left\{(x-y)|x-y|^{-(N+\alpha)}\right\}=\alpha |x-y|^{-(N+\alpha)}$.
On the other hand, \cite{Abatangelo} introduces the notion of nonlocal directional curvatures
---neither used in this paper.

The nonlocal mean curvature is the Euler-Lagrange equation for the fractional peri\-meter functional,
as first discovered in \cite{Caffarelli2010}; see also \cite{Figalli2014}. 
Nonlocal minimal surfaces are hypersurfaces with zero NMC and were introduced in 2009 by Caffarelli, Roquejoffre,
and Savin~\cite{Caffarelli2010}. They are the limiting configurations of some fractional diffusions 
---for instance, of the fractional Allen-Cahn equation; see the exposition in \cite{Valdinoci} and its references. 
The seminal paper \cite{Caffarelli2010} established the first
existence and regularity theorems. Within these years, there have been important efforts and results concerning
nonlocal minimal surfaces, the following being the main ones. Here we refer to minimizing nonlocal minimal surfaces ---a smaller class than
that of surfaces with zero NMC. When $\alpha$ is close to 1, Caffarelli and Valdinoci~\cite{Caffarelli2011A,Caffarelli2011B}
proved their $C^{1,\gamma}$ regularity
up to dimension $N\leq 7$. This, together with the subsequent results of Barrios, Figalli, and Valdinoci~\cite{Barrios}, 
leads, for $\alpha$ close to 1, to their $C^{\infty}$ regularity up to dimension $N\leq 7$. 
For $N=2$ and any $\alpha$, Savin and Valdinoci~\cite{Savin2012} established  that they are
$C^{\infty}$. Figalli and Valdinoci \cite{Figalli2013} have proved that if they are Lipschitz in $\R^N$, 
then they are $C^\infty$. Finally, D\'avila, del Pino, and Wei \cite{Davila2014B} initiate the important study of 
nonlocal minimal cones in any dimension, characterizing the stability or instability of $\alpha$-Lawson cones.
Besides, they also construct surfaces of revolution with zero NMC, for instance the fractional catenoid. 
Still, apart from dimension $N=2$, there is a lot to be understood,  mainly for the classification of all 
stable nonlocal minimal cones.

Instead, to our knowledge, there are no works on CNMC hypersurfaces, that is, hypersurfaces with constant nonlocal
mean curvature. The purpose of this article is twofold. We establish results both of Alexandrov and of Delaunay type
for the nonlocal mean curvature.

Our first result is the nonlocal or fractional counterpart of the classical result by Alexandrov \cite{A62} 
on the characterization of spheres as the only closed embedded CMC-hypersurfaces. The precise statement
is the following.

\begin{theorem}
\label{sec:introduction-2}
Suppose that $E$ is a nonempty bounded open set with $C^{2,\beta}$-boundary for some $\beta >\alpha$ and with the property that $H_E$ is constant on $\partial E$. Then $E$ is a ball. 
\end{theorem}

After completing our proof, we have learnt that this result has also been established,
at the same time and independently of ours, by Ciraolo, Figalli, Maggi, and Novaga~\cite{Ciraolo2015} ---a paper in which they also prove
stability results with respect to this rigidity theorem.

Our proof of the result relies on the moving planes method introduced by Alexandrov in \cite{A62}.
We also use a formula for the tangential derivatives of the nonlocal mean curvature, 
Proposition~\ref{sec:introduction-1} below. In contrast with the classical case, 
there is no local comparison principle related to the fractional mean curvature. However, nonlocal 
elliptic operators enjoy (by their definition) a very strong global comparison principle. One of its
consequences is the following. 
Connectedness of $E$ is obviously a necessary assumption in classification results for 
CMC-hypersurfaces, whereas the assumptions of Theorem~\ref{sec:introduction-2} allow for disconnected 
sets $E$ with finitely many connected components. Related to this result is the one in \cite{Fall-Jarohs} where fractional overdetermined 
problems are studied on smooth bounded open sets using moving plane techniques and  comparison principles. 
See also \cite{Dalibard--Gerard-Varet}.

The second purpose of our paper is to establish a nonlocal analogue in the plane of the classical result of
Delaunay \cite{Delaunay} on periodic cylinders with constant mean curvature.
We study sets $E\subset\R^2$ with constant nonlocal mean curvature
which have the form of bands or ``cylinders'' in the plane
$$
E=\{(s_1,s_2)\in\R^2: -u(s_1)<s_2<u(s_1)\},
$$
where $u:\R\to(0,\infty)$ is a positive function. In contrast with classical mean curvature,
we note that a straight band $\{-R<s_2<R\}$ has \emph{positive} constant nonlocal mean curvature. 
This can be easily seen either from \eqref{eq:def-frac-curvature} or from \eqref{geometric}.

We establish the existence of a continuous family of bands which are periodic in the first variable~$s_1$, have
all the same constant nonlocal mean curvature, and converge to the straight band. The width $2R$ of the straight
band will be chosen so that the periods of the new bands converge to $2\pi$ as the bands tend to $\{-R<s_2<R\}$.
Our result is of perturbative nature and thus 
we find periodic bands which are all very close to the straight one. 

Therefore we show that, in the nonlocal setting, these objects already exist in dimension 2 ---while they only exist
in dimensions 3 and higher in the classical CMC setting. Our precise result is the following.

\begin{theorem}
\label{res:cyl1}
For every $\alpha\in(0,1)$ there exist $R>0$ and a small $\nu>0$, both depending only on $\alpha$, 
and a continuous family of periodic functions $u_a:\R\to\R$ parameterized by 
$a\in(-\nu,\nu)$, for which the sets
$$
E_a=\{(s_1,s_2)\in\R^2: -u_{a}(s_1)<s_2<u_{a}(s_1)\}
$$ 
have all the same nonlocal mean curvature, equal to a constant $h_R>0$ depending only on $\alpha$,
and converge to the straight band $\{-R<s_2<R\}$ as $a\to 0$.
Moreover, $E_a\not= E_{a'}$ for $a\not= a'$, $u_a$ has minimal period ${2\pi}/{\lambda(a)}$ if $a\not= 0$, 
and the periods
$$
2\pi/\lambda(a)\to 2\pi \qquad\text{as } a\to 0.
$$

In addition, $u_{a}$ is of class $C^{1,\beta}(\R)$ for some $\beta\in (\alpha,1)$, even 
with respect to $s_1=0$, and of the form
\begin{equation}\label{form}
  u_{a}(s)=R+\frac{a}{\lambda(a)}\left\{\cos\left(\lambda(a)s\right)+v_{a}(\lambda(a)s)\right\}
\end{equation}
with $v_{a}\to 0$ as $a\to 0$ in the norm of $C^{1,\beta}(\R)$, 
and with $v_{a}=v_{a}(\sigma)$ satisfying $\int_0^{2\pi}v_{a}(\sigma)\cos(\sigma)\, d\sigma=0$.
\end{theorem}

The smoothness of these curves, and in general of graphs in $\R^N$ with constant NMC, follows
from the methods and results of Barrios, Figalli, and Valdinoci~\cite{Barrios} on nonlocal minimal graphs.

In a forthcoming work, \cite{Cabre2015A}, we treat this Delaunay type result in $\R^N$.
In  \cite{Cabre2015A} we will also show further properties of the CNMC curves $s_2=u_a(s_1)$; for instance, 
for $a>0$, they are decreasing in $(0,\pi/\lambda(a))$ (the first half of their period).

For a related nonlocal equation, but different than nonlocal mean curvature,
the recent paper \cite{Davila2015} by D\'avila, del Pino, Dipierro, and Valdinoci establishes variationally
the existence of periodic and cylindrical symmetric hypersurfaces. These surfaces minimize a 
certain fractional perimeter under a volume constraint.

Regarding nonlocal minimal curves in the plane, let us mention here that Savin and Valdinoci \cite{Savin2012}
have proved that the only minimizing nonlocal minimal curves in the plane are the straight lines.  
A minimizing nonlocal minimal curve is the boundary of a set that
minimizes fractional perimeter in compact sets (and, in particular, has zero NMC). Another important rigidity theorem is that
of Figalli and Valdinoci \cite{Figalli2013} establishing that, in $\R^{3}$, nonlocal minimal graphs of functions 
$u:\R^{2}\to\R$ are necessarily planes.

To prove Theorem~\ref{res:cyl1} we use the Lyapunov-Schmidt procedure. We were inspired by the results on periodic
solutions in Ambrosetti and Prodi~\cite{Ambrosetti}, as well as in our work in a simpler setting (the 
semilinear one) that we are carrying out in \cite{Cabre2015B}. The first step in the proof, as in the case of 
nonlocal minimal cones in \cite{Davila2014B} and 
nonlocal minimal graphs in \cite{Barrios}, is to write the NMC operator acting on graphs 
of functions ---the functions $u_a$ above. The result is a nonlinear fractional operator of quasilinear type acting on functions
$u=u(s)$  ---see the operators \eqref{graphNMC}, \eqref{maineq2}, \eqref{Phi10}, 
and \eqref{Phi220} below. Its linearization gives rise to 
the fractional Laplacian $(-\Delta)^{(1+\alpha)/2}$ and some convolution operators. We use 
both Fourier series and H\"older spaces to treat all these operators. Let us mention that the implicit function theorem
in weighted H\"older spaces has already been used in the important study of nonlocal minimal
cones by D\'avila, del Pino, and Wei \cite{Davila2014B}.

Related to our work, the nice papers by Sicbaldi~\cite{Sicbaldi2010} and by Schlenk and Sicbaldi~\cite{Sicbaldi2012}
establish the existence of periodic and cylindrical symmetric domains in $\R^N$ whose first Dirichlet
eigenfunction has constant Neumann data on the boundary. This is therefore a
nonlinear nonlocal operator ``based'' on the half-Laplacian or Dirichlet to Neumann map. Their papers also
use the Lyapunov-Schmidt procedure.

The articles \cite{Dipierro,Davila2014A} also use the Lyapunov-Schmidt method, here for some 
semilinear fractional elliptic equations.

Our paper is organized as follows. The following two sections concern the Alexandrov type result. Section~2
treats the moving planes method, while section~3 establishes a formula for the tangential derivatives of
the NMC. Sections~4 to 6 are dedicated to the Delaunay type result. The first one sets up the 
nonlinear nonlocal operator to be studied, as well as the Lyapunov-Schmidt procedure
and functional spaces to be used. Section~5 is devoted to the study of the linearization of our
nonlinear problem, while section~6 establishes the $C^1$ character of our nonlinear operators.

\section{The moving plane argument}
\label{sec:introduction}

This section is devoted to the proof of Theorem~\ref{sec:introduction-2}. We need the following proposition
which might also be useful for other problems. It gives a formula for the tangential derivatives along $\partial E$
of the NMC.

\begin{proposition}
\label{sec:introduction-1}
If $E \subset \R^N$ is bounded and $\partial E$ is of class $C^{2,\beta}$ for some $\beta> \alpha$, then $H_E$ is of class $C^1$ on $\partial E$, and we have 
$$
\partial_v H_E(x) = (N+\alpha)\, PV \int_{\R^N} \tauE(y)  |x-y|^{-(N+2+\alpha)} (x-y) \cdot v \,dy
$$
for $x \in \partial E$ and $v \in T_x \partial E$.
\end{proposition}

Here, $T_x \partial E$ denotes the tangent space to $\partial E$ at $x$.
Since the proof of this proposition is somewhat technical, we postpone it to section~3. 
With the help of this proposition, we may now complete the

\begin{proof}[Proof of Theorem~\ref{sec:introduction-2}.]
Since the fractional mean curvature is invariant under reflections, translations and rotations, 
it suffices to show that every set $E \subset \R^N$ satisfying the assumptions of Theorem~\ref{sec:introduction-2} 
is convex and symmetric in $x_1$ after a translation in the $x_1$-direction. To prove this, we introduce some notation. 
As usual, we let $\nu$ denote the unit normal vector field on $\partial E$ pointing outside $E$. 

For $\lambda \in \R$, we let $Q_\lambda: \R^N \to \R^N $ denote the reflection with respect to the hyperplane $\{x_1 = \lambda\}$. Moreover, we put $E^\lambda:= Q_\lambda(E)$ and 
$$
E^\lambda_-:= \{x \in E^\lambda \::\: x_1 < \lambda \}.
$$
Note that $E^\lambda_- = \varnothing$ if $\lambda>0$ is sufficiently large, whereas $E^\lambda_- =E^\lambda \not \subset E$ if $-\lambda$ is sufficiently large.  We define
$$
\lambda^*: = \inf \{ \lambda \in \R\::\: \text{$E^\mu_- \subset E$ for $\mu \ge \lambda$, and $\nu_1(x)>0$ for all $x \in \partial E$ with $x_1>\lambda$} \}.
$$

We claim that 
\begin{equation}
  \label{eq:final-claim}
E = E^{\lambda^*} .
\end{equation}
This will complete the proof of the theorem. To prove the claim, we first note that 
\begin{equation}
  \label{eq:trivial-implication}
E^{\lambda^*}_- \subset E,   
\end{equation}
since $\partial E$ is of class $C^2$.  Moreover, by standard arguments detailed e.g. in \cite[Section 5.2]{FR00}
(see also a sketch of the argument in Remark~\ref{cases12} below),
there exists a point $x \in \partial E$ having one of the following properties:\\[0.2cm]
{\em Case 1 (interior touching):} $x_1 < \lambda^*$ and $x \in \partial E^{\lambda^*}$. \\[0.2cm]
{\em Case 2 (non-transversal intersection):} $x_1=\lambda^*$ and $e_1:= (1,0,\dots,0) \in T_x E$.
\\[0.2cm]
We treat these cases separately, and we adjust the notation first. For simplicity, we will write $\lambda$ instead of 
$\lambda^*$ and $Q$ instead of $Q_\lambda=Q_{\lambda^*}$ in the following. Obviously, we have that 
\begin{equation}
  \label{eq:tau-1}
  \tauE \equiv \tau_{E^\lambda} \qquad \text{in $E \cap E^\lambda$ and in $\R^N \setminus (E \cup E^\lambda)$,}
\end{equation}
whereas 
\begin{equation}
  \label{eq:tau-2}
  \tauE \equiv -\tau_{E^\lambda} \equiv 1 \; \text{in $E \setminus E^\lambda$}\qquad \text{and}\qquad 
 \tau_{E^\lambda} \equiv - \tauE  \equiv 1 \; \text{in $E^\lambda  \setminus E$}.
\end{equation}
We now let $\cH$ denote the half space $\{x_1 < \lambda\}$, and we note that, by (\ref{eq:trivial-implication})
---recall that $\lambda=\lambda^*$---,
$$
E \setminus E^\lambda \subset \cH \qquad \text{and}\qquad E^\lambda  \setminus E = Q(E \setminus E^\lambda) \subset  Q(\cH).
$$

Finally, for $\eps>0$ and $x\in\R^{N}$ the point in either Case 1 or Case 2, we set
$$
A^1_\eps:= \{y \in E \setminus E^\lambda\::\: |y-x| \ge \eps\} \subset \cH \quad \text{and}\quad  A^2_\eps:= \{y \in E^\lambda \setminus E\::\: |y-x| \ge \eps\} \subset  Q(\cH).
$$

We first consider {\bf Case 1:}  From (\ref{eq:tau-1}) and (\ref{eq:tau-2}), it then follows that
\begin{align*}
0 &=\frac{1}{2} \Bigl(H_{E}(Q(x)) - H_E(x)  \Bigr)= \frac{1}{2} \Bigl(H_{\mbox{\tiny $E^\lambda$}}(x) - H_E(x)  \Bigr)\\
&= \frac{1}{2} PV  \int_{\R^N} \bigl( \tauE(y)-\tau_{\mbox{\tiny $E^\lambda$}}(y)\bigr)
|x-y|^{-(N+\alpha)}\,dy\\
 &= \lim_{\eps \to 0} \Bigl( \int_{A^1_\eps} |x-y|^{-(N+\alpha)}\,dy - \int_{A^2_\eps} |x-y|^{-(N+\alpha)}\,dy\Bigr)\\
 &= \lim_{\eps \to 0} \int_{A^1_\eps} |x-y|^{-(N+\alpha)}\,dy - \int_{E^\lambda \setminus E} |x-y|^{-(N+\alpha)}\,dy.
\end{align*}
Here we have used that $\partial E$ and $\partial E^{\lambda}$ are sets of measure zero since $\partial E$ is $C^{2}$, and
the fact that the second integral in the last line exists since $x \in \cH$ and $E^\lambda  \setminus E \subset Q(\cH)$.  Thus, by monotone convergence we have 
$$
\int_{E^{\lambda}\setminus E} |x-y|^{-(N+\alpha)}\,dy = 
\lim_{\eps \to 0} \int_{A^1_\eps} |x-y|^{-(N+\alpha)}\,dy = \int_{E\setminus E^{\lambda}} |x-y|^{-(N+\alpha)}\,dy
$$
and therefore  
\begin{align*}
0&=  \int_{E \setminus E^\lambda} |x-y|^{-(N+\alpha)}\,dy - \int_{E^\lambda \setminus E} |x-y|^{-(N+\alpha)}\,dy\\&= \int_{E \setminus E^\lambda} \Bigl( |x-y|^{-(N+\alpha)} - |x-Q(y)|^{-(N+\alpha)} \Bigr)\, dy.
\end{align*}
Since $E \setminus E^\lambda \subset \cH$ and $|x-y| < |x -Q(y)|$ for $y \in \cH$, we deduce that $|E \setminus E^\lambda|= 0$ and thus $E$ equals $E^\lambda$ up to a set of measure zero. Since both $E$ and $E^\lambda$ are bounded sets with $C^2$ boundary, we conclude that $E= E^\lambda$, as claimed in (\ref{eq:final-claim}).

We now consider {\bf Case 2:} Since $e_1:= (1,0,\dots,0) \in T_x E \cap T_x E^\lambda$ and $H_E$ (respectively, $H_{E^\lambda}$) are constant functions on $\partial E$ (respectively, $\partial E^\lambda$),  we have $\partial_{e_1} H_E(x) =\partial_{e_1} H_{E^\lambda}(x)=0$. Applying 
Proposition~\ref{sec:introduction-1} and recalling (\ref{eq:tau-1}) and (\ref{eq:tau-2}), 
we find that 
\begin{align*}
0 &= \frac{1}{2}\bigl(\partial_{e_1} H_{E}(x)- \partial_{e_1} H_{E^\lambda}(x) \bigr)\\
 &= \frac{1}{2}(N+\alpha) PV \int_{\R^N} (x_1-y_1) \bigl(\tauE(y)- \tau_{\mbox{\tiny $E^{\lambda}$}}(y)\bigr)   |x-y|^{-(N+2+\alpha)} \,dy\\
 &= (N+\alpha) \lim_{\eps \to 0} \Bigl( \int_{A^1_\eps} (\lambda-y_1) |x-y|^{-(N+2+\alpha)}\,dy - \int_{A^2_\eps} (\lambda-y_1) |x-y|^{-(N+2+\alpha)}\,dy\Bigr)\\
&= (N+\alpha) \lim_{\eps \to 0}  \Bigl( \int_{A^1_\eps} |\lambda-y_1| |x-y|^{-(N+2+\alpha)}\,dy + \int_{A^2_\eps} |\lambda-y_1| |x-y|^{-(N+2+\alpha)}\,dy\Bigr).  
\end{align*}
Hence it follows that 
$$
\lim_{\eps \to 0} \int_{A^1_\eps} |\lambda-y_1| |x-y|^{-(N+2+\alpha)}\,dy= 0.
$$
Since $A^1_{\eps'}  \supset A^1_{\eps}$ for  $0< \eps' \le  \eps$ and the integrand is a continuous positive function in $A^1_{\eps}
\subset \cH$, this implies that $|A^1_\eps| = 0$ for every $\eps>0$ and thus $|E \setminus E^\lambda|= 0$. As in Case~1, we conclude that $E= E^\lambda$, as claimed in (\ref{eq:final-claim}).
\end{proof}

\begin{remark}\label{cases12}
{\rm
Here we sketch the argument showing that either Case 1 or 2 in the previous proof must happen
when $\lambda=\lambda^*$.

Let us denote points by $x=(x_1,x')$ and, for every $\lambda$, $T^\lambda:=\{x'\in\R^{N-1} : (\lambda,x')\in E\}
\subset\R^{N-1}$. By the definition of $\lambda^*$ it is clear that, for every $x'\in T^{\lambda^*}$, the set
$\{x_1\in\R : x_1>\lambda^* \text{ and } (x_1,x')\in E\}$ is an open interval. Therefore, since in addition
$\nu_1 (x)>0$ for all $x\in\partial E$ with $x_1>\lambda^*$ (by definition of $\lambda^*$), we have that
\begin{equation}\label{graphlam}
\partial E\cap\{x_1>\lambda^*\}=\{(\va^{\lambda^*}(x'),x') : x'\in T^{\lambda^*}\}
\end{equation}
for a $C^1$ function $\va^{\lambda^*}$ on the open set $T^{\lambda^*}\subset \R^{N-1}$. The previous statement follows
from the inverse function theorem.

Suppose now that Case 2 did not happen. Then, we would have $\nu_1 (x)>0$ for all $x\in\partial E$ with 
$x_{1}\geq \lambda^*$. Thus, by the inverse function theorem, $\va^{\lambda^*}$ would be $C^1$ up to the boundary
$\partial T^{\lambda^*}$. The implicit function theorem also gives that \eqref{graphlam} would also hold with
$\lambda^*$ replaced by $\mu$ if $\mu\in (\lambda^*-\delta,\lambda^*)$ for some small $\delta >0$
---with $\va^{\lambda^*}$ replaced by a new $C^{1}$ function $\va^{\mu}$. 

Hence, for a smaller $\eta>0$, the reflection of $\{E\cap\{\mu<x_1<\lambda^*+\eta\}\}$ with respect to $\{x_1=\mu\}$ 
is still contained in $E$, if $\mu\in (\lambda^*-\eta,\lambda^*)$. Now, the reflection of
$K:=\overline E\cap \{x_1\geq \lambda^*+\eta\}$ with respect to $\{x_1=\lambda^*\}$ 
is a compact set contained in $\overline E$. If we assume that Case 1 neither happens, then
$K\cap\partial E=\varnothing$ and thus $K$ is at a positive distance from $\partial E$. By continuity, it follows that
all reflections of $K$ with respect to $\{x_1=\mu\}$ are also contained in $E$ for 
$\mu\in (\lambda^*-\eta',\lambda^*)$ (for a perhaps smaller $\eta'>0$). This is a contradiction, since then
$E^\mu_-\subset E$ for $\mu\in (\lambda^*-\eta',\lambda^*)$, contradicting the definition of $\lambda^*$.
}
\end{remark}

\section{Tangential derivatives of the NMC. Proof of Proposition~\ref{sec:introduction-1}}\label{s:app:de-curv}
In this section we prove Proposition~\ref{sec:introduction-1}. Let $E \subset \R^N$ be bounded and such that $\partial E$ is of class $C^{2,\beta}$ for some $\beta> \alpha$. As before, we let $\nu$ denote the unit normal vector field on $\partial E$ pointing outside $E$.

\begin{lemma}
\label{sec:append-proof-lemma-epsilon}
For $\eps>0$ sufficiently small, the function 
$$
H_\eps : \R^N \to \R, \qquad H_\eps(x)= -\int_{|z| \ge \eps } \tauE(x+z)|z|^{-(N+\alpha)}\,dz 
$$
is of class $C^1$ with 
\begin{equation}
\frac{\partial H_\eps}{\partial x_i} (x) = - (N+\alpha) \int_{|z| \ge \eps}   \tauE(x+z) |z|^{-(N+2+\alpha)} z_i \,dz + \eps^{-(N+1+\alpha)} \int_{|z|= \eps} \tauE(x+z)z_i  \,d\sigma(z) \label{eq:deriv-H_eps}
\end{equation}
for $i=1,\dots,N$.
\end{lemma}

Note that the boundary integral in (\ref{eq:deriv-H_eps}) vanishes if 
$\partial B_\eps(x) \subset E$ or $\partial B_\eps(x) \subset E^c$.

\begin{proof}
Since $E \subset \R^N$ is bounded and $\partial E$ is of class $C^{2,\beta}$, we have, for $\eps>0$ sufficiently small,
\begin{equation}
  \label{eq:hausdorff-intersection}
  \cH^{N-1}(\partial B_\eps(x) \cap \partial E) = 0 \qquad \text{for all $x \in \R^N$}
\end{equation}
---that we will assume from now on--- and also that the boundary integral in (\ref{eq:deriv-H_eps}) is continuous in $x$.

Let $\tau_n \in C^1_c( \R^N)$, $n \in \N$, be chosen such that 
\begin{equation}
  \label{eq:bound-tau-n}
|\tau_n| \le 1 \quad \text{in $\R^N$ for all $n \in \N$}  
\end{equation}
 and such that 
\begin{equation}
  \label{eq:tau-n-uniform}
\tau_n \to \tauE \qquad \text{uniformly on compact subsets of $\R^N \setminus \partial E$.}
\end{equation}
Moreover, for $n \in \N$, we consider 
$$
h^n :\R^N  \to \R, \quad h^n(x)= \int_{|z| \ge \eps} \tau_n(x+z) |z|^{-(N+\alpha)}\,dz. 
$$
By a standard application of Lebesgue's theorem, $h^n$ is of class $C^1$, and by integration by parts we have 
\begin{align*}
\frac{\partial h^n}{\partial x_i} (x)&= 
\int_{|z| \ge \eps} \frac{\partial \tau_n}{\partial x_i} (x+z) \, |z|^{-(N+\alpha)}\,dz\\
&=  (N+\alpha) 
\int_{|z| \ge \eps}   \tau_n(x+z)    |z|^{-(N+2+\alpha)} \, z_i \,dz
- \eps^{-(N+1+\alpha)} \int_{|z| = \eps}  \tau_n(x+z)  z_i  \,d\sigma(z).
\end{align*}

Let $R>0$. From (\ref{eq:bound-tau-n}) and (\ref{eq:tau-n-uniform}), it easily follows that 
$$
\sup_{|x| \le R} \int_{\eps \le |z| \le \rho} |\tau_n(x+z)-\tau_{E}(x+z)| d z\to 0  \qquad \text{as $n \to \infty$ for every $\rho>\eps$,}
$$
and by using (\ref{eq:bound-tau-n}) again this implies that   
\begin{equation}
  \label{eq:uniform-first-integral}
\sup_{|x| \le R} \Bigl| \int_{|z| \ge \eps} \bigl(\tau_n(x+z)-\tau_{E}(x+z)\bigr)     |z|^{-(N+2+\alpha)} \, z_i \,dz \Bigr| \to 0 \qquad 
\text{as $n \to \infty$}
\end{equation}
for $i=1,\dots,N$. A similar argument, using (\ref{eq:hausdorff-intersection}), (\ref{eq:bound-tau-n}) and (\ref{eq:tau-n-uniform}), shows that 
\begin{equation}
  \label{eq:uniform-second-integral}
\sup_{|x| \le R}  \Bigl|  \int_{|z| = \eps}  \bigl(\tau_n(x+z)-\tau_{E}(x+z)\bigr)  z_i  \,dz\Bigr| \to 0 \qquad \text{as $n \to \infty$}
\end{equation}
for $i=1,\dots,N$. Since also $h^n \to - H_\eps$ as $n \to \infty$ uniformly in $B_R(0)$, it follows that $H_\eps$ is of class $C^1$ in $B_R(0)$ with partial derivatives given by (\ref{eq:deriv-H_eps}). Since $R>0$ was arbitrary, the claim follows.
\end{proof}

In the following, we set
$$
A(\eps',\eps):= \{z \in \R^N\::\: \eps' < |z| < \eps\}\qquad \text{for $0< \eps' \le \eps$.}
$$

\begin{lemma}
\label{sec:sketch-argument-lemma-2}  
There exist $\eps_0>0$ and $C>0$ with the following property. For every $x \in \partial E$, $v \in T_x \partial E$ and $0< \eps' < \eps < \eps_0$ we have
\begin{equation}
  \label{eq:rel-est-A_eps_eps}
\Bigl|
\int_{A(\eps',\eps)} \tauE(x+z)   |z|^{-(N+2+\alpha)}\, z \cdot  v \,dz \Bigr| \le 
C|v|  \eps^{\beta-\alpha}
\end{equation}
and 
\begin{equation}
  \label{eq:rel-est-S_eps}
\Bigl|
\int_{|z|= \eps} \tauE(x+z)  z \cdot  v \,d\sigma(z) \Bigr| \le
C|v| \eps^{N+1+\beta}. 
\end{equation}
\end{lemma}

\begin{proof}
Without loss of generality, we may assume that $x=0$, $v=e_1$ and $\nu(0)=e_N$. 
For $\eps>0$, we consider the sets 
$$
B_\eps^{N-1}:= \{y \in \R^{N-1}\::\:|y| < \eps\} \qquad \text{and} 
\qquad C_\eps:= B_\eps^{N-1} \times (-\eps,\eps) \subset \R^N,
$$
so that we have the inclusions 
$$
A(\eps',\eps) \subset B_\eps(0) \subset C_\eps \qquad \text{for $0< \eps' \le \eps$.}
$$
Since $\partial E$ is of class $C^{2,\beta}$, there exists $\eps_0 \in (0,1]$ and a $C^{2,\beta}$-function $h : B_{\eps_0}^{N-1} \to \R$ with $h(0)=0$, $\nabla h(0)=0$ and such that 
\begin{align*}
&C_{\eps_0} \cap \partial E = \{(y,h(y))\::\: y \in B_{\eps_0}^{N-1}\};\\
&C_{\eps_0} \cap E = \{(y,t) \::\: y \in B_{\eps_0}^{N-1},\; t < h(y)\};\\
&C_{\eps_0} \cap E^c = \{(y,t) \::\: y \in B_{\eps_0}^{N-1},\; t \geq h(y)\}.
\end{align*}
Moreover, making $\eps_0$ smaller if necessary, we find $c>0$ such that 
\begin{equation*}
\Bigl|h (y)-q(y)\Bigr| \le c |y|^{2+\beta} \qquad \text{for $y \in B_{\eps_0}^{N-1}$}\quad \text{with $q(y):= \sum \limits_{i,j=1}^N \partial_{ij}h(0)y_i y_j.$}  
\end{equation*}

Next, for $0< \eps' \le \eps\le \eps_0$, we split $A(\eps',\eps)$ into the subsets
\begin{align*}
A^+(\eps',\eps)&:= \{(y,t) \in A(\eps',\eps)\,:\, t > q(y)+c|y|^{2+\beta}\} \;\subset \;E^c,\\ 
A^-(\eps',\eps)&:= \{(y,t) \in A(\eps',\eps)\,:\, t < q(y)-c|y|^{2+\beta}\} \;\subset \;E,\\ 
A^0(\eps',\eps)&:= \{(y,t) \in A(\eps',\eps)\,:\, q(y)-c|y|^{2+\beta} \le t \le q(y)+c|y|^{2+\beta}\}.
\end{align*}
Since the sets $A^\pm(\eps',\eps)$  are invariant under the reflection $(y,t) \mapsto (-y,t)$ and
$\mp \tauE \equiv 1$ in $A^\pm(\eps',\eps)$, we find that 
$$
\int_{A^\pm(\eps',\eps)}\tauE (z)  |z|^{-(N+2+\alpha)}z_1 \,dz= \mp \int_{A^\pm(\eps',\eps)} |z|^{-(N+2+\alpha)}z_1 \,dz = 0.
$$
We thus have that 
\begin{align*}
\Bigl|
\int_{A(\eps',\eps)}& \tauE(z)|z|^{-(N+2+\alpha)}z_1 \,dz\Bigr| \le 
\int_{A^0(\eps',\eps)} |z|^{-(N+1+\alpha)} \,dz\\
&\hspace{-.5cm} \le \int_{B_{\eps}^{N-1}} \int_{q(y)-c|y|^{2+\beta}}^{q(y)+c |y|^{2+\beta}} |(y,t)|^{-(N+1+\alpha)}  \,dt  \,dy \le \int_{B_{\eps}^{N-1}} |y|^{-(N+1+\alpha)} \int_{q(y)-c|y|^{2+\beta}}^{q(y)+c |y|^{2+\beta}}   \,dt \,dy\\
&\hspace{-.5cm} =2c \int_{B_{\eps}^{N-1}}|y|^{-N+1+\beta-\alpha} dy
= 2c  \:\omega_{\text{\tiny $N\!\!-\!\!2$}} \int_0^{\eps} r^{\beta-\alpha-1}dr =  \frac{2c\, \omega_{\text{\tiny $N\!\!-\!\!2$}}}{\beta-\alpha}  \eps^{\beta-\alpha},
\end{align*}
where $\omega_{\text{\tiny $N\!\!-\!\!2$}}$ denotes the surface area of the unit sphere in $\R^{N-1}$.

To see (\ref{eq:rel-est-S_eps}), we split $\partial B_\eps(0) \subset \R^N$ into the subsets
\begin{align*}
S^+(\eps)&:= \{(y,t) \in \partial B_\eps(0) \,:\, t > q(y)+c|y|^{2+\beta}\} \;\subset \;E^c,\\ 
S^-(\eps)&:= \{(y,t) \in \partial B_\eps(0) \,:\, t < q(y)-c|y|^{2+\beta}\} \;\subset \;E,\\ 
S^0(\eps)&:= \{(y,t) \in \partial B_\eps(0) \,:\, q(y)-c|y|^{2+\beta} \le t \le q(y)+c|y|^{2+\beta}\}.
\end{align*}
Since the sets $S^\pm(\eps)$  are invariant under the reflection $(y,t) \mapsto (-y,t)$ and
$\mp \tauE \equiv 1$ on $S^\pm(\eps)$, we find that 
$$
\int_{S^\pm(\eps)}\tauE (z)  z_1 \,d\sigma(z)= \mp \int_{S^\pm(\eps)} z_1 \,d \sigma(z) = 0
$$
and therefore 
\begin{equation}
  \label{eq:S-estimate}
\Bigl| \int_{\partial B_\eps(0)} \tauE (z)  z_1 \,d\sigma(z) \Bigr| \le 
\int_{S^0(\eps)} |z| \,d\sigma(z) = \eps \cH^{N-1}(S^{0}(\eps)) = \eps^N \cH^{N-1}(S^1(\eps))
\end{equation}
with 
$$
S^1(\eps):= \{(y,t) \in \partial B_1(0) \,:\, \eps q(y)-c\eps^{1+\beta} |y|^{2+\beta} \le t \le \eps q(y)+c \eps^{1+\beta} |y|^{2+\beta}\}.
$$ 

To estimate $\cH^{N-1}(S^1(\eps))$, we fix $\bar q \ge  \max_{\partial B_1(0)}|q|$ and recall that $\eps \le \eps_0 \le 1$. Thus for $(y,t) \in S^1(\eps)$ we have $|t| \le \eps (\bar q + c)$,  and hence 
$$
\bigl|q(\frac{y}{|y|})-q(y)\bigr|= \bigl|q(\frac{y}{|y|})\bigr| (1- |y|^2)=  \bigl|q(\frac{y}{|y|})\bigr| t^2 \le \eps^2 \bar q (\bar q+ c)^2
 = c_1 \eps^2 
$$
with $c_1 = \bar q (\bar q+ c)^2$ and therefore, with $c_2:= c_1 +c$, 
\begin{align*}
\eps q(\frac{y}{|y|})- c_2 \eps^{1+\beta}  \le \eps q(\frac{y}{|y|}) - c_1 \eps^3 - c \eps^{1+\beta} \le  \eps q(y)-c\eps^{1+\beta} \le t &\le \dots\\
&\le \eps q(\frac{y}{|y|})+ c_2 \eps^{1+\beta}.
\end{align*}
Denoting by $S^{N-2}$ the unit sphere in $\R^{N-1}$, we may thus estimate  
$$
\cH^{N-1}(S^1(\eps)) \le  \int_{S^{N-2}} \int_{\eps q(\rho)- c_2 \eps^{1+\beta}}^{\eps q(\rho)+ c_2 \eps^{1+\beta}} (1-t^2)^{\frac{N-3}{2}}\,dt \,d\rho.
$$
If $N \ge 3$, this implies that  
$$
\cH^{N-1}(S^1(\eps)) \le \int_{S^{N-2}} 2c_2 \eps^{1+\beta} \,d\rho =  2 c_2\, \omega_{\text{\tiny $N\!\!-\!\!2$}}\eps^{1+\beta},
$$
whereas in the case $N=2$ we have 
$$
(1-t^2)^{\frac{N-3}{2}} \le (1- \eps^2 (\bar q + c)^2)^{-\frac{1}{2}} \le 2 
\qquad \text{for $(y,t) \in S^1(\eps)$ if $\eps \le \frac{\sqrt{3}}{2(\bar q + c)}$.}
$$
and therefore 
$$
\cH^{N-1}(S^1(\eps)) \le \int_{S^{N-2}} 4c_2 \eps^{1+\beta} \,d\rho =  4 c_2\, \omega_{\text{\tiny $N\!\!-\!\!2$}}\eps^{1+\beta} \qquad \text{if $\eps \le \frac{\sqrt{3}}{2(\bar q + c)}.$}
$$

Combining this with (\ref{eq:S-estimate}) and assuming without loss that $\eps_0 \le \frac{\sqrt{3}}{2(\bar q + c)}$, we find in both cases that  
\begin{equation*}
\Bigl| \int_{\partial B_\eps(0)} \tauE (z)  z_1 \,d\sigma(z) \Bigr| \le 4 c_2\, \omega_{\text{\tiny $N\!\!-\!\!2$}}\, \eps^{N+1+\beta} \qquad \text{for $0 <\eps \le \eps_0$.}
\end{equation*}
Since $\partial E$ is compact and of class $C^{2+\beta}$,  the constants $\eps_0, c, \bar q$ can be chosen independently of $x \in \partial E$, and then also $c_1,c_2>0$ do not depend on $x \in \partial E$.  Hence both (\ref{eq:rel-est-A_eps_eps}) and (\ref{eq:rel-est-S_eps}) hold with $C:=  \omega_{\text{\tiny $N\!\!-\!\!2$}} \max \{\frac{2c}{\beta-\alpha}, 4 c_2 \}$.
\end{proof}

\begin{proof}[Proof of Proposition~\ref{sec:introduction-1} (completed)]
Let $(\eps_k)_k$ be a decreasing sequence of positive numbers with $\eps_k \to 0$. For $k \in \N$, we consider the functions
$$
H_E^k : \partial E \to \R, \qquad H_E^k(x)= - \int_{|z| \ge \eps_k } \tauE(x+z)|z|^{-(N+\alpha)}\,dy. 
$$
We then have that  
$$
\lim_{k \to \infty} H_E^k(x) = - \lim_{\eps_{k} \to 0}\int_{|x-y| \ge \eps_k } \tauE(y)|x-y|^{-(N+\alpha)}\,dy = H_E(x) \qquad \text{for $x \in \partial E$.}
$$
Moreover, we may pass to a subsequence such that all functions $H_E^k$ are of class $C^1$ on $\partial E$ by Lemma~\ref{sec:append-proof-lemma-epsilon}. Moreover,  for a given $C^1$-vector field $v$ on $\partial E$ and $k \in \N$, we have $\partial_{v(x)} H_E^k(x)  = \rho^k(x) + \sigma^k(x)$ for $x \in \partial E$ with $\rho^k,\sigma^k : \partial E \to \R$ given by 
$$
\rho^k(x):=
-(N+\alpha) \int_{|z| \ge \eps_{k}} \tauE(x+z)|z|^{-(N+2+\alpha)}z \cdot v(x)   \,dz
$$
and
$$
\sigma^k(x):=
\eps_{k}^{-(N+1+\alpha)} \int_{|z|=\eps_{k}} \tauE(x+z) z \cdot v(x) d\sigma(x).
$$

Since $\beta>\alpha$, Lemma~\ref{sec:sketch-argument-lemma-2} implies that $(\rho^k)_k$, $k \in \N$ is a Cauchy sequence in $C(\partial E)$, whereas $\sigma^k \to 0$ uniformly on $\partial E$ as $k \to \infty$. Moreover, a standard sequence mixing argument shows that, independently of the choice of the sequence $(\eps_k)_k$, we have 
$$
\lim_{k \to \infty} \rho^k(x) = -(N+\alpha) \lim_{\eps \to 0} 
\int_{|z| \ge \eps} \tauE(x+z)|z|^{-(N+2+\alpha)}z \cdot v(x)   \,dz
\qquad \text{for $x \in \partial E$.} 
 $$
Since the vector field $v$ was chosen arbitrarily, it thus follows that $H_E$ is of class $C^1$ on $\partial E$ with 
\begin{align*}
\partial_{v(x)}  H_E(x) &= \lim_{k \to \infty} \rho^k(x) =  -(N+\alpha) \lim_{\eps \to 0} 
\int_{|z| \ge \eps} \tauE(x+z)|z|^{-(N+2+\alpha)}z \cdot v(x)   \,dz\\
&=  (N+\alpha) \lim_{\eps \to 0}  \int_{|x-y| \ge \eps} \tauE(y)|x-y|^{-(N+2+\alpha)}(x-y) \cdot v(x)   \,dz,
\end{align*}
as claimed.
\end{proof}

\section{Periodic CNMC curves in $\R^{2}$}

In this section we set up the Lyapunov-Schmidt procedure to prove Theorem~\ref{res:cyl1}
on Delaunay-type curves with CNMC in the plane, that is, curves with constant nonlocal
mean curvature. The full proof of the theorem will be completed 
in this and the following two sections.

The following easy lemma gives a formula for the nonlocal mean curvature of a set given by
$\{-u(s_1)<s_2<u(s_1)\}$ in terms of the positive function $u=u(s)$. The same computation already
appears in \cite{Barrios,Davila2014B}.

\begin{lemma}
\label{res:cyl2}
If $E=\{(s_1,s_2)\in\R^2 : -u(s_1)<s_2<u(s_1)\}$, where $u:\R\to(0,\infty)$ is a function of class 
$C^{1,\beta}$ for some $\beta\in(\alpha,1)$ such that $0<m_1\leq u\leq m_2$ for two positive constants $m_i$, 
its nonlocal mean curvature $H_E$ ---that we will denote by $H(u)$---
at a point $(s,u(s))$ is given by
\begin{equation}\label{graphNMC}
\begin{split}
 \frac{1}{2}H(u)(s)& =\int_\R F\left(\dfrac{u(s)-u(s-t)}{|t|}\right)\ d\mu(t) \\
   & \hspace{1cm}-\int_\R\left\{F\left(\dfrac{u(s)+u(s-t)}{|t|}\right)-F(+\infty)\right\}\ d\mu(t),
\end{split}
\end{equation}
where the integrals are to be understood in the principal value sense,
\begin{equation}\label{measmu}
d\mu(t)=\frac{dt}{|t|^{1+\alpha}},
\end{equation}
and
\begin{equation}\label{funF}
  F(q)=\int_0^q\dfrac{d\tau}{(1+\tau^2)^{\frac{2+\alpha}{2}}}.
\end{equation}
\end{lemma}

In the fractional perimeter functional (whose Euler-Lagrange equation
is the NMC operator), the first integral in \eqref{graphNMC} 
corresponds to the interactions of points in the upper curve
$\{s_2=u(s_1)\}$ of $\{-u(s_1)<s_2<u(s_1)\}$ with points in the same upper curve. 
The second integral corresponds to interactions of points in the upper curve with points in the lower disjoint
curve $\{s_2=-u(s_1)\}$.

\begin{proof}[Proof of Lemma \ref{res:cyl2}]
We have that 
$$
-H(u)(s)=\int_{\R^2}\dfrac{1_E((s_1,s_2))-1_{E^c}((s_1,s_2))}{|(s,u(s))-(s_1,s_2)|^{2+\alpha}}\ ds_1\, ds_2
=\int_\R I(s,s_1)\,ds_1,
$$
where $1_A$ denotes the characteristic function of $A$, and
\begin{eqnarray*}
I(s,s_1)&=&\left\{\int_{-u(s_1)}^{u(s_1)} ds_2-\int_{-\infty}^{-u(s_1)} ds_2-\int_{u(s_1)}^{+\infty} ds_2\right\}
\left((s-s_1)^2+(s_2-u(s))^2\right)^{-\frac{2+\alpha}{2}}\\
&=&\left\{2\int_{-u(s_1)}^{u(s_1)}\, ds_2-\int_\R\, ds_2\right\}|s-s_1|^{-(2+\alpha)}
\left(1+\left(\dfrac{s_2-u(s)}{|s-s_1|}\right)^2\right)^{-\frac{2+\alpha}{2}}.
\end{eqnarray*}
Then, with the new variable $\tau=(s_2-u(s))/|s-s_1|$, one gets 
$$
\hspace{-3.7cm}
I(s,s_1)=|s-s_1|^{-(1+\alpha)}\left(2\int_{\frac{-u(s_1)-u(s)}{|s-s_1|}}^{\frac{u(s_1)-u(s)}{|s-s_1|}}
\ d\tau-\int_\R\ d\tau\right)(1+\tau^2)^{-\frac{2+\alpha}{2}}$$
$$
\hspace{.1cm}
=|s-s_1|^{-(1+\alpha)}\left\{2\left(F\left(\frac{u(s_1)-u(s)}{|s-s_1|}\right)-F\left(\frac{-u(s_1)-u(s)}{|s-s_1|}\right)\right)-F(+\infty)+F(-\infty)\right\}$$
$$
\hspace{-2.05cm}
=-2|s-s_1|^{-(1+\alpha)}\left\{F\left(\frac{u(s)-u(s_1)}{|s-s_1|}\right)-
F\left(\frac{u(s)+u(s_1)}{|s-s_1|}\right)+F(+\infty)\right\}.
$$
 Changing the variable $s_1=s-t$, we arrive at the expression of the lemma.
\end{proof}

With the use of formula \eqref{graphNMC} we can compute the NMC ---that we denote by $h_R$---
of the straight band of width $u_R\equiv 2R$:
\begin{equation}\label{NMCR}
H(u_R)=-2\int_\R\left\{F\left(\frac{2R}{|t|}\right)-F(+\infty)\right\} d\mu(t)=:h_R>0.
\end{equation}

Note that the functions $u_a=u_a(s)$ in \eqref{form},
$$
 u_{a}(s)=R+\frac{a}{\lambda}\{\cos\left(\lambda s\right)+v_{a}(\lambda s)\}, 
 \quad\text{ with } \lambda=\lambda(a),
$$
have a period that may change with $a$. It will be very convenient, in order to apply the implicit function theorem, 
to work with functions all with the same period $2\pi$. Thus, we rescale the variables $s=s_1$ and $u(s)=s_2$
and see how the NMC changes after rescaling. Since the NMC is a geometric quantity, the following comes 
as no surprise. 
If, for $\lambda>0$, $w^\lambda(s):=\lambda w(s/\lambda)$, then with the change of variables $\tau=t/\lambda$ we see that
\begin{eqnarray*}
&& \hspace{-1.5cm}  H(w^\lambda)(s)= 2\int_\R F\left(\dfrac{\lambda w(s/\lambda)-\lambda w((s-t)/\lambda)}{|t|}\right)\ d\mu(t) \\
 &-& 2\int_\R\left\{F\left(\dfrac{\lambda w(s/\lambda)+\lambda w((s-t)/\lambda)}{|t|}\right)-F(+\infty)\right\}\ d\mu(t)=
   \lambda^{-\alpha} H(w)(s/\lambda).
\end{eqnarray*}
Thus, we must look for functions $u=(u_a)^\lambda$ of the form
\begin{equation}\label{express}
  u(s)=\lambda R+a\{\cos(s)+v_{a}(s)\},
\end{equation}
with $v_a$ even, $2\pi$-periodic, and orthogonal to $\cos(\cdot)$ in $[0,\pi]$. In addition, since we want the NMC for
the band given by $u_a$ to be $h_R$, we need that $H(u)=\lambda^{-\alpha} h_R$. Thus, from \eqref{NMCR}
and making below the change of variables $ t= \overline t/\lambda$, we see that $u$ must
satisfy
\begin{eqnarray*}
 \frac{1}{2} H(u) &=&  \frac{1}{2} \lambda^{-\alpha} h_R = - \lambda^{-\alpha}
 \int_\R\left\{F(2R/|t|)-F(+\infty)\right\} d\mu(t) \\
 &= & \lambda^{-\alpha}
 \int_\R   \int_{2R/|t|}^{+\infty}  (1+\tau^2)^{-\frac{2+\alpha}{2}} \, d\tau\, \frac{dt}{|t|^{1+\alpha}} 
 = \int_\R   \int_{2\lambda R/|\overline t|}^{+\infty}  (1+\tau^2)^{-\frac{2+\alpha}{2}} \, d\tau\, 
 \frac{d\overline t}{|\overline t|^{1+\alpha}} \\
 &=& \int_\R   d\mu(\overline t)  \int_{2\lambda R/|\overline t|}^{+\infty}  (1+\tau^2)^{-\frac{2+\alpha}{2}} 
 \, d\tau = \int_\R   d\mu(t)  \int_{2\lambda R/|t|}^{+\infty}  (1+\tau^2)^{-\frac{2+\alpha}{2}} 
 \, d\tau\\
&=& -   \int_\R\left\{F\left(2\lambda R/|t|\right)-F(+\infty)\right\} d\mu(t).
\end{eqnarray*}
Finally, recalling the expression \eqref{graphNMC} for $H(u)$, we conclude that the function $u$ given by 
\eqref{express} must satisfy the equation
\begin{equation}\label{maineq1}
\int_\R \left\{F\left(\dfrac{u(s)-u(s-t)}{|t|}\right)-\left\{F\left(\dfrac{u(s)+u(s-t)}{|t|}\right)-
F\left(\dfrac{2\lambda R}{|t|}\right)\right\}\right\} d\mu(t) =0.
\end{equation}

Recall that we look for solutions of this equation of the form
\begin{equation*}
 u(s)=\lambda R+a\{\cos(s)+v_{a}(s)\} =\lambda R + a \va(s).
\end{equation*}
In our proof, after a Lyapunov-Schmidt reduction, we will apply the implicit function theorem at 
$a=0$, $\lambda=1$, and $v_a=0$ ---taking $v_a$ orthogonal to $\cos(\cdot)$ in $L^2(0,\pi)$. 
This will give us $\lambda$ and $v_a$ as functions of $a$.  
To have the linearized operator to be invertible and be able to use directly the implicit function theorem, 
it is necessary to divide equation \eqref{maineq1} by $a$ (as in \cite{Cabre2015B} and as in the classical paper 
by Crandall-Rabinowitz~\cite{Crandall1971}) and work with the new operator 
\begin{equation}\label{maineq2}
 \begin{split}
& \hspace{-.2cm}  \Phi(a,\lambda,\va)(s):= \int_\R \dfrac{1}{a}F\left(a\ \dfrac{\va(s)-\va(s-t)}{|t|}\right) d\mu(t) \\
  & - \int_\R \dfrac{1}{a}\left\{F\left(\dfrac{2\lambda R+a(\va(s)+\va(s-t))}{|t|}\right)
   -F\left(\dfrac{2\lambda R}{|t|}\right)\right\}d\mu(t).
 \end{split}
\end{equation}
We need to solve the nonlinear equation
$$
 \Phi(a,\lambda,\va)=0.
$$

Let us introduce the functional spaces in which we work.
We take $\beta$ such that
\begin{equation}\label{defbeta}
0<\alpha <\beta <\min\{1,2\alpha+1/2\}.
\end{equation}
The condition $\beta < 2\alpha+1/2$ is technical (to simplify a proof on regularity) and could
be avoided. Consider the spaces
\begin{equation}\label{spX}
 X:=C^{1,\beta}_{p,e}=\{\va:\R\to\R\, :\ \va\in C^{1,\beta}(\R) \text{ is $2\pi$-periodic and even}\}
\end{equation}
and
\begin{equation}\label{spY}
 Y:=C^{0,\beta-\alpha}_{p,e}=\{\tilde\va:\R\to\R\, :\ \tilde\va\in C^{0,\beta-\alpha}(\R) 
 \text{ is $2\pi$-periodic and even}\}.
\end{equation}
Note that if $\va\in X$, then 
$$
\va'(0)=\va'(\pi)=0.
$$
Since the functions in these spaces are even and $2\pi$-periodic, we can take as norms $\Vert\cdot\Vert_X$
and $\Vert\cdot\Vert_Y$, respectively, the $C^{1,\beta}([0,\pi])$ and $C^{0,\beta-\alpha}([0,\pi])$ norms.
That is,
\begin{equation}\label{normX}
\Vert\va\Vert_X:= \Vert\va\Vert_{L^\infty(0,\pi)} + \Vert\va'\Vert_{L^\infty(0,\pi)}
+\sup_{0\leq s<\overline s\leq\pi}
\frac{|\va'(s)-\va'(\overline s)|}{|s-\overline s|^\beta}
\end{equation}
and
$$
\Vert\tilde\va\Vert_Y:= \Vert\tilde\va\Vert_{L^\infty(0,\pi)} +\sup_{0\leq s<\overline s\leq\pi}
\frac{|\tilde\va(s)-\tilde\va(\overline s)|}{|s-\overline s|^{\beta-\alpha}}.
$$

We must study the operator 
$$
\Phi(a,\lambda,\varphi)=\Phi_1 (a,\varphi)-\Phi_2(a,\lambda,\varphi),
$$
where
\begin{equation*}
 \Phi_1 (a,\varphi):=\int_\R \frac{1}{a} F(a\delta_- \varphi)\ d\mu(t),
\end{equation*}
\begin{equation}\label{Phi2}
 \Phi_2(a,\lambda,\varphi):=\int_\R \frac{1}{a} \left\{F\left( \frac{2\lambda R}{|t|}+
 a\frac{\delta_0\varphi}{|t|}\right)-F\left(\frac{2\lambda R}{|t|}\right)\right\} d\mu(t),
\end{equation}
and we define
$$
\delta_- \varphi (s,t):= \frac{\varphi(s)-\varphi(s-t)}{|t|},
$$
\begin{equation}\label{deltaplus}
 \delta_+ \varphi (s,t):= \frac{\varphi(s)-\varphi(s+t)}{|t|}
\end{equation}
($\delta_+$ does not appear above, but it will be used later), and
$$
\delta_0 \varphi (s,t):= \varphi(s)+\varphi(s-t).
$$

The previous operators ---and their equivalent expressions below--- are seen to be well defined 
(some in the principal value sense) in section~6, where we will establish much more: their differentiable
character between the spaces $X$ and $Y$ above.

To express the first term $\Phi_1$ of the operator in a nicer manner, observe that 
$$
 F_1(a,q):=\dfrac{1}{a}F(aq)=\dfrac{1}{a}\int_0^{aq}\dfrac{d\rho}{(1+\rho^2)^{\frac{2+\alpha}{2}}}=
 \int_0^q\dfrac{d\tau}{(1+a^2\tau^2)^{\frac{2+\alpha}{2}}}
$$
is a smooth function both of $a$ and $q$, and also that $F_1(0,q)=q$. 
Thus, we have that 
\begin{equation}\label{Phi10}
 \Phi_1 (a,\varphi):=\int_\R F_1(a,\delta_- \varphi)\ d\mu(t),
\end{equation}
where
\begin{equation}\label{F1}
F_1(a,q):=\int_0^q \frac{d\tau}{(1+a^2\tau^2)^{\frac{2+\alpha}{2}}}.
\end{equation}

Next, we deal with the second term $\Phi_2$ in the operator $\Phi$. Recall
that it is given by expression \eqref{Phi2}.
Using the change of variables $\tau= (2\lambda R+a\delta_0\va\,\overline\tau)/|t|$,
we see that
\begin{eqnarray*}
& &\hspace{-1.5cm}\frac{1}{a} \left\{F\left( (2\lambda R+
 a\delta_0\varphi)/|t|\right)-F\left(2\lambda R/|t|\right)\right\}
 = \frac{1}{a}\int_{2\lambda R/|t|}^{(2\lambda R+a\delta_0\va)/|t|}
 \frac{d\tau}{(1+\tau^2)^{\frac{2+\alpha}{2}}}  \\
 &=& \frac{\delta_0\va}{|t|} \int_0^1 \left\{ 
 1+\left((2\lambda R+a\delta_0\va\, \overline\tau)/|t| \right)^2\right\}^{-\frac{2+\alpha}{2}}
d\overline\tau \\
&=&|t|^{1+\alpha} \delta_0\va \int_0^1 \frac{d\overline\tau}{\left\{ t^2+(2\lambda R+ a\delta_0\va
\,\overline\tau)^2\right\}^{\frac{2+\alpha}{2}}}.
\end{eqnarray*}
Therefore,
\begin{equation}\label{Phi220}
\Phi_2(a,\lambda,\va)=\int_\R \delta_0\va \, F_2(t,a,\lambda,\delta_0\va)\, dt,
\end{equation}
where
\begin{equation}\label{F30}
 F_2(t,a,\lambda,r)= \int_0^1 \frac{d\overline\tau}{\left\{ t^2+(2\lambda R+ ar
\overline\tau)^2\right\}^{\frac{2+\alpha}{2}}}.
\end{equation}
Note that the measure in \eqref{Phi220} is $dt$ and not $d\mu(t)$ as above for $\Phi_1$, and that $\delta_0$ 
is not an increment but a sum which is not divided by $|t|$.

Let us write our operator $\Phi$ when $a=0$. The following expression can be computed from \eqref{maineq2}
differentiating with respect to $a$ at $a=0$, or can be obtained from \eqref{Phi10} and \eqref{Phi220} 
---taking into account that
$(t^2+(2\lambda R)^{2})^{-\frac{2+\alpha}{2}}=F'(2\lambda R/|t|)/|t|^{2+\alpha}$. We have
\begin{equation}\label{azero}
 \hspace{-.17cm} \Phi(0,\lambda,\va)(s)=\int_\R\left(\va(s)-\va(s-t)\right)\dfrac{dt}{|t|^{2+\alpha}}
  -\int_\R\left(\va(s)+\va(s-t)\right)F'\left(\dfrac{2\lambda R}{|t|}\right)\dfrac{dt}{|t|^{2+\alpha}}, 
\end{equation}
where, as always, the integrals are to be understood in the principal value sense.

We choose now $R>0$ such that 
$\Phi(a=0,\lambda=1,\va=\cos(\cdot))=0$ holds, i.e., such that
\begin{equation}\label{defR}
\begin{split}
   0&= \Phi(0,1,\cos(\cdot))(s)=\int_\R\left(\cos(s)-\cos(s-t)\right)\dfrac{dt}{|t|^{2+\alpha}}\\
  &\hspace{3.5cm}-\int_\R\left(\cos(s)+\cos(s-t)\right)F'\left(\dfrac{2\lambda R}{|t|}\right)\dfrac{dt}{|t|^{2+\alpha}}
\end{split}
\end{equation}
for all $s\in\R$. We will see in Lemma \ref{existsR} below that there exists a unique $R>0$, depending only on
$\alpha$, such that \eqref{defR} holds for all $s\in\R$.

The next task is to study the linearized operator of $\Phi$ at $(a=0,\lambda=1,\va=\cos(\cdot))$. This
will be done in all detail in next section. Let us say here that we will see that the function 
$$
\cos(\cdot) \text{ belongs to the kernel of } D_\va\Phi(0,1, \cos(\cdot)), 
$$
and this is why we must use a Lyapunov-Schmidt
reduction corresponding to the following subspaces of $X$ and $Y$.

In $L^2(0,\pi)$ we consider the orthogonal basis given by $\{1,\cos(\cdot),\cos(2\cdot),\cos(3\cdot),\ldots\}$ and
the subspaces
\begin{equation}\label{subsp}
V_1=\langle\cos(\cdot)\rangle \qquad\text{and} \qquad V_2= \langle\cos(\cdot)\rangle^\bot
 =\langle 1,\cos(2\cdot),\cos(3\cdot),\ldots \rangle.
\end{equation}
We also consider
$$
X_1=X\cap V_1, \qquad X_2=X\cap V_2, \qquad Y_1=Y\cap V_1, \qquad Y_2=Y\cap V_2,
$$
meaning for the first intersection, for instance, the functions defined in $\R$ belonging to~$X$ 
such that, when restricted to $(0,\pi)$, belong to $V_1$. Denote  the standard projections by
$$
\Pi_1:Y\to Y_1 \qquad\text{and}\qquad \Pi_2:Y\to Y_2.
$$
Retaking our notation above, we look for functions $\va$ of the form
$$ 
\va(s)=\cos(s)+v(s), \quad\text{ with } v\in X_2.
$$
We then write our operator acting on the function $v\in X_2$ as
$$
\overline\Phi:A\subset\R\times\R\times X_2\to Y, \qquad \overline\Phi(a,\lambda,v):=\Phi (a,\lambda,\cos(\cdot)+v),
$$
where $A$ will be an appropriate open set containing $(0,1,0)$.

The key result to apply the implicit function theorem is the following proposition. 
It will be proved in the next two sections. The differentiability properties stated in the proposition are
proved in section~6 and ensure, in particular, that all operators and integrals above are well defined.
Recall that, by \eqref{defR}, we have
that $\overline\Phi(0,1,0)=0$.

\begin{proposition}\label{propPhi}
There exists $\nu>0$ small enough $($depending only on $\alpha)$ for which the operator 
$\overline\Phi=\overline\Phi (a,\lambda,v): (-\nu,\nu)\times (1/2,3/2)\times B_{1}(0)\subset \R\times\R\times X_2
\to Y$ is of class $C^1$. 
Moreover, the linear operator 
$$
(D_\lambda\overline\Phi,D_v\overline\Phi)(0,1,0):\R\times X_2\to Y
$$
is continuous and invertible.
\end{proposition}

From this result and the implicit function theorem, Theorem \ref{res:cyl1} follows immediately.

\begin{proof}[Proof of Theorem~\ref{res:cyl1}]
As mentioned above, we have $\overline\Phi(0,1,0)=0$. Proposition~\ref{propPhi} allows to apply the 
implicit function theorem and obtain the family of functions $u_a$ in the statement of the theorem. 
All properties stated in the theorem follow immediately from the setting described in this section,
with the exception of the following two statements that need to be justified.

It remains to prove that the minimal period of $u_a$ is
$2\pi/\lambda(a)$ if $a\not=0$, and that $u_a\not\equiv u_{a'}$ if $a\not=a'$. Let us start from the first statement.
Clearly, it is equivalent to prove that, after the rescaling, the function 
$$
u(s)=\lambda R+ a\{\cos(s)+v_a(s)\},
$$
with $a\not=0$ and $v_a$ orthogonal to $\cos(\cdot)$ in $L^2(0,\pi)$, has minimal period $2\pi$. 
This is an easy task, by
expressing $v_a(s)$ as a Fourier series $a_0+\sum_{k=2}^{\infty}a_k\cos(ks)$. Now, if $T$ is the minimal
period of $u$, we must have 
\begin{eqnarray*}
&& \hspace{-1cm}\cos(s)+v_a(s)=\cos(s+T)+v_a(s+T)\\
 &=&\cos(s)\cos(T)-\sin(s)\sin(T)
+a_0+\sum_{k=2}^{\infty}a_k\{\cos(ks)\cos(kT)-\sin(ks)\sin(kT)\}.
\end{eqnarray*}
Multiplying the first and last terms in the above expression by $\cos(s)$
and integrating in $(0,2\pi)$, we deduce that $\cos(T)=1$. Hence the minimal period is $T=2\pi$.

Finally, from this we can easily deduce that $u_a\not\equiv u_{a'}$ if $a\not=a'$. Indeed,
if $u_a\equiv u_{a'}$ then their minimal periods would agree, and thus $\lambda(a)=\lambda(a')$.
Now, this leads to $a=a'$,
since $u_a(s)=R+\frac{a}{\lambda(a)}\{\cos(\lambda(a)s)+v_a(\lambda(a)s)\}$ and $v_a(\sigma)$ is orthogonal to
$\cos(\sigma)$ in $L^2(0,\pi)$.
\end{proof}

\section{The linearized operator acting on even periodic functions}

In this section we study the linearized operator of $\Phi$ at $(0,1,\cos(\cdot))$ and we establish
the last statement on invertibility in Proposition~\ref{propPhi}. For this, the main results that we prove
are collected in the following result.

\begin{proposition}
\label{res:linear}
We have that
\begin{equation}\label{dlam}
 D_\lambda\Phi(0,1,\cos(\cdot))=\gamma \cos(\cdot)
\end{equation}
for some constant $\gamma>0$. On the other hand, for all $\psi\in X$,
\begin{equation}\label{dv}
 L\psi:=D_\varphi \Phi(0,1,\cos(\cdot))\, \psi=C^{-1}_{(1+\alpha)/2}\, (-\Delta)^{\frac{1+\alpha}{2}}\psi-
 (\textstyle{\int_\R} P_R)\, \psi-P_R*\psi,
\end{equation}
where $C_{(1+\alpha)/2}$ is the usual constant factor in the definition of 
$(-\Delta)^{\frac{1+\alpha}{2}}$ and 
$$
P_R(t)=\frac{1}{|t|^{2+\alpha}}F'\left(\frac{2R}{|t|}\right)=
\frac{1}{\left\{(2R)^2+|t|^2\right\}^{\frac{2+\alpha}{2}}}.
$$

The functions $e_k(s)=\cos(ks)$, $k=0,1,2,3\dots$, are all eigenfunctions of $L$, with eigenvalues satisfying
 \begin{equation}\label{eigen}
 Le_k=\lambda_k  e_k,\qquad \ \lambda_0<0=\lambda_1<\lambda_2<\lambda_3<\cdots\quad\text{ and }
 \end{equation}
 \begin{equation}\label{assym}
  \dfrac{\lambda_k}{k^{1+\alpha}}\to\mu_\infty>0 \qquad \text{as } k\to\infty.
\end{equation}
As a consequence, we will have that the linear operator 
$$
(D_\lambda\overline\Phi,D_v\overline\Phi)(0,1,0):\R\times X_2\to Y
$$
is continuous and invertible.
\end{proposition}

The rest of this section is dedicated to prove the above proposition and, in particular, 
the last statement on invertibility on H\"older spaces in Proposition~\ref{propPhi}.

From \eqref{azero}, we compute
\begin{equation}\label{Dlambda}
\begin{split}
  D_\lambda\overline\Phi(0,1,0)(s) &=-\int_\R\left(\cos(s)+\cos(s-t)\right)F''(2R/|t|)
  \, 2R\,\frac{dt}{|t|^{3+\alpha}}\\
  &=-\cos(s)\int_\R\left(1+\cos(t)\right)F''(2R/|t|)\,2R\,\frac{dt}{|t|^{3+\alpha}}.
\end{split}
\end{equation}
The last simplification comes from the fact that $\cos(s-t)=\cos(s)\cos(t)+\sin(s)\sin(t)$, 
$\sin(t)$ is an odd function of $t$, and $F''(2R/|t|)|t|^{-3-\alpha}$ is an even function of $t$. 
Note also that the last integral converges at $t=0$, since 
$F''(2R/|t|)\sim|t|^{3+\alpha}$ near $t=0$. 
Observe also that this integral is a strictly negative number, since $1+\cos(t)\ge 0$ and 
$F''(2R/|t|)<0$.

From \eqref{Dlambda} we deduce that 
$$
D_\lambda(\Pi_1\overline\Phi)(0,1,0)(s)=-\cos(s)\int_\R\left(1+\cos(t)\right)
F''(2R/|t|)\,2R\,\frac{dt}{|t|^{3+\alpha}}
$$ 
and that this integral is negative, and also that 
$$
D_\lambda(\Pi_2\overline\Phi)(0,1,0)\equiv 0.
$$

As a consequence of these two statements on $D_\lambda(\Pi_i\overline\Phi)(0,1,0)$, 
to establish the last statement in Proposition~\ref{res:linear} it only remains to prove that 
$D_v(\Pi_2\overline\Phi)(0,1,0)$ is an isomorphism between $X_2$ and $Y_2$.
But, from \eqref{azero}, and with $P_R$ as in the statement of Proposition~\ref{res:linear}, we have
\begin{eqnarray*}
    Lw (s):&=&D_v\overline\Phi(0,1,0)\, w(s)\\
    &=&\int_\R(w(s)-w(s-t))\frac{dt}{|t|^{2+\alpha}}-
    \int_\R (w(s)+w(s-t))F'\left(\frac{2R}{|t|}\right)\frac{dt}{|t|^{2+\alpha}}\\
    &=& \int_\R(w(s)-w(s-t))\frac{dt}{|t|^{2+\alpha}}-(\textstyle{\int_\R} P_R)\, w(s)-(P_R*w)(s)\\
    &=&\left\{ C^{-1}_{(1+\alpha)/2}\, (-\Delta)^{\frac{1+\alpha}{2}}w-(\textstyle{\int_\R} P_R)w-P_R*w
    \right\} (s).
\end{eqnarray*}
We emphasize that $P_R\in (C^\infty\cap L^1\cap L^\infty)(\R)$ and that $P_R(t)$ is an even function of $t$.

The following simple result states that the functions $\cos(k\cdot)$ are always eigenfunctions of
any convolution operator with an even kernel.

\begin{lemma}
\label{res:eigenfunctions}
{\rm i)} If $P\in L^1(\R)$ is an even function and $e_k(s)=\cos(ks)$, $k=0,1,2,3\dots$, then 
$$
(P*e_k)(s)=(\textstyle{\int_\R}\cos(kt)P(t)\, dt)\ e_k(s)
$$
for all $s\in\R$. In particular, $e_k$ is an eigenfunction of the operator $P*\cdot$.

{\rm ii)} In a similar way, the functions $e_k(s)=\cos(ks)$ are eigenfunctions of $(-\Delta)^{\frac{1+\alpha}{2}}$.
\end{lemma}

\begin{proof}
i) This is a simple calculation:
\begin{eqnarray*}
(P*e_k)(s)&=&\textstyle{\int_\R}\cos(k(s-t))P(t)\ dt=\textstyle{\int_\R}\left\{\cos(ks)\cos(kt)+\sin(ks)
\sin(kt)\right\}P(t)\,dt\\
&=&\left(\textstyle{\int_\R}\cos(kt)P(t)\ dt\right)\cos(ks),
\end{eqnarray*}
since $P$ is even.

ii) For the fractional Laplacian, in the principal value sense, we have
$$
\int_\R\dfrac{\cos(ks)-\cos(ks-kt)}{|t|^{2+\alpha}}\ dt=\int_\R\dfrac{\cos(ks)-\cos(ks)\cos(kt)-
\sin(ks)\sin(kt)}{|t|^{2+\alpha}}\ dt,$$
and this last integral, in the principal value sense, is equal to 
$$
\left(\int_\R\dfrac{1-\cos(kt)}{|t|^{2+\alpha}}\ dt\right) \cos(ks),
$$
as desired.
\end{proof}

Thus,
\begin{equation*}
  Le_k=\left\{\int_\R\dfrac{1-\cos(kt)}{|t|^{2+\alpha}}\ dt-\int_\R\dfrac{1+\cos(kt)}{|t|^{2+\alpha}}
  F'\left(\dfrac{2R}{|t|}\right) dt\right\}e_k=:\lambda_ke_k, 
\end{equation*}
for $k=0,1,2,\dots$.
Note that $\lambda_0=-\int_\R 2|t|^{-2-\alpha}F'(2R/|t|)\, dt<0$. 
Now, 
$$
\lambda_1=\int_\R\frac{1-\cos(t)}{|t|^{2+\alpha}}\ dt-\int_\R\frac{1+\cos(t)}{|t|^{2+\alpha}}
F'\left(\frac{2R}{|t|}\right)\ dt=0,
$$
since this value is the same as the factor multiplying $\cos(s)$ in expression \eqref{defR}, defining $R$, and that we next take to be equal to zero. 
Thus, $L\cos(\cdot)=Le_1=\lambda_1e_1\equiv 0$.

Let us see now that there exists a unique $R>0$ for which $\lambda_1=0$.

\begin{lemma}
\label{existsR}
There exists a unique $R>0$ ---which depends only on $\alpha$--- such that 
\begin{equation}\label{eqR}
  \lambda_1(R)=\int_\R\dfrac{1-\cos(t)}{|t|^{2+\alpha}}\ dt-\int_\R\dfrac{1+\cos(t)}{|t|^{2+\alpha}}
  F'\left(\dfrac{2R}{|t|}\right)\ dt=0.
\end{equation}
\end{lemma}

\begin{proof}
First, 
$$
\lambda'_1(R)=-\int_\R\dfrac{1+\cos(t)}{|t|^{2+\alpha}}\ \dfrac{2}{|t|} \
F''\left(\dfrac{2R}{|t|}\right)\ dt>0,
$$ 
and therefore $\lambda_1$ is increasing in $R$. Second, applying the monotone convergence theorem for integrals to
\begin{equation*}
\lambda_1(R)=\int_\R\left\{\frac{1-\cos(t)}{|t|^{2+\alpha}}-\frac{1+\cos(t)}{|t|^{2+\alpha}}
\left(1+\frac{(2R)^2}{|t|^2}\right)^{-\frac{2+\alpha}{2}}\right\}\, dt,
\end{equation*}
we deduce that $\lambda_1(R)\nearrow\int_\R\frac{1-\cos(t)}{|t|^{2+\alpha}}\ dt>0$ as $R\to+\infty$, and also that  
$\lambda_1(R)\searrow \int_\R-\frac{2\cos(t)}{|t|^{2+\alpha}}\ dt=-\infty$ as $R\to 0$. Thus, the result is proved.
\end{proof}

Let us see now that the sequence $\lambda_k$ is increasing in $k$, 
and that therefore we have $\lambda_0<0=\lambda_1<\lambda_2<\lambda_3<\cdots$.
Indeed, using the change of variables $\overline t=kt$,
\begin{eqnarray*}
\lambda_k&=&\int_\R\left\{(1-\cos(kt))-(1+\cos(kt))F'(2R/|t|)\right\}\ \frac{dt}{|t|^{2+\alpha}}\\
&=&k^{1+\alpha}\int_\R\left\{(1-\cos(\overline t))-(1+\cos(\overline t))\left(1+\dfrac
{4R^2}{|\overline t/k|^2}
\right)^{-\frac{2+\alpha}{2}}\right\}\dfrac{d\overline t}{|\overline t|^{2+\alpha}}=k^{1+\alpha}\mu_k,
\end{eqnarray*}
where we have defined
$$
\mu_k=\int_\R\left\{(1-\cos(\overline t))-(1+\cos(\overline t))
\left(1+\dfrac{4R^2k^2}{|\overline t|^2}\right)^{-\frac{2+\alpha}{2}}\right\}
\dfrac{d\overline t}{|\overline t|^{2+\alpha}}.
$$
But both $k^{1+\alpha}$ and $\mu_k$ are increasing in $k$. And, by the monotone convergence theorem, 
$$
-\int_\R(1+\cos(\overline t))\left(1+\frac{4R^2k^2}{|\overline t|^2}\right)^{-\frac{2+\alpha}{2}}
\dfrac{d\overline t}{|\overline t|^{2+\alpha}} \nearrow 0
$$
as $k\to+\infty$. Hence $\mu_k \to\int_\R\frac{1-\cos(\overline t)}{|\overline t|^{2+\alpha}}\ 
d\overline t=:\mu_\infty>0$. 
Thus, and this will be crucial in the next argument,
\begin{equation*}
  \dfrac{\lambda_k}{k^{1+\alpha}}\to\mu_\infty>0 \qquad\text{as } k\to +\infty.
\end{equation*}

From \eqref{eigen}, $L\cos(\cdot)=0$, the asymptotics \eqref{assym}, 
and the definition of fractional Sobolev spaces in terms of Fourier coefficients,
we deduce that 
$$
L_{|V_2}=\Pi_2 L_{|V_2}: H^{1+\alpha}_{p,e}\cap V_2\to L^2_{p,e}\cap V_2 \quad\text{ is 
continuous and invertible},
$$
where $H^{1+\alpha}_{p,e}$ denotes the space of functions which are even and $2\pi$ periodic in $\R$
and belong to the Sobolev space $H^{1+\alpha} $ in bounded sets of $\R$. The same definition applies
to $L^2_{p,e}$.

To complete the proof of Proposition~\ref{res:linear}, we 
need to show that $L$ is also continuous and invertible from
$X_2$ onto $Y_2$. The fact that $L$ sends $X_2$ to $Y$ continuously (and thus into $Y_2$), 
follows from $L_{|X_2}=D_v\overline\Phi(0,1,0)$ and 
the first part of Proposition~\ref{propPhi} (to be proved in next section) stating that 
$\overline\Phi$ is $C^1$ from its domain in $\R\times\R\times X_2$ into $Y$.

It remains to establish that $L$ is invertible in these spaces. 
For this, given $f\in Y_2$, since then $f\in L^{2}_{p,e}\cap V_2$,
we know that we can find a unique $w\in H^{1+\alpha}_{p,e}\cap V_2$ (using Fourier series and the eigenvalues
$\lambda_k$ above) such that $Lw=f$. Recall that $L$ is given by 
$$
Lw=c_1 (-\Delta)^{\frac{1+\alpha}{2}} w - c_2 w -P_R * w
$$
for some positive constants $c_i$. Hence, $Lw=f$ can be written as
\begin{equation}\label{invert}
c_1 (-\Delta)^{\frac{1+\alpha}{2}}w = c_2 w +P_R * w + f \quad\text{ in }\R.
\end{equation} 
Since $w\in H^{1+\alpha}_{p,e}$ and $P_R$ is smooth and integrable, 
we have that $P_R*w\in H^{1+\alpha}_{p,e}$. Use now that
$H^{1+\alpha}_{p,e}\subset Y=C^{0,\beta-\alpha}_{p,e}$ by Morrey's embedding, since $1+\alpha-1/2=
1/2+\alpha>\beta-\alpha$ ---recall \eqref{defbeta}.

Therefore, the right hand side of \eqref{invert} belongs to $Y=C^{0,\beta-\alpha}_{p,e}$. 
By standard H\"older regularity for the fractional Laplacian ---see Proposition~2.8 of \cite{Sil}---,
it follows that $w\in X= C^{1+\alpha,\beta}_{p,e}$.

\section{Differentiability properties of the nonlinear operator acting on even periodic functions}

To prove the differentiability properties of the operator $\overline\Phi=\overline\Phi(a,\lambda,v)$ 
stated in Proposition~\ref{propPhi}, it is sufficient to establish that 
$$
\Phi=\Phi(a,\lambda,\va): (-\nu,\nu)\times(1/2,3/2)\times B_{10}(0)\subset \R\times\R\times X\to Y
$$
is of class $C^1$. Here one should recall that $\va=\cos(\cdot)+v$ and note that, by \eqref{normX}, 
$\Vert\cos(\cdot)\Vert_X\leq 1+1+\pi<9$. 

We start studying the first term, $\Phi_1$, of the operator ---which is the most delicate. It turns out
to be a nonlinear version of the fractional Laplacian $(-\Delta)^{\frac{1+\alpha}{2}}$. More precisely,
by expression \eqref{Phi12} below, it is a quasilinear version of the fractional Laplacian ---the
second order increments in \eqref{Phi12} are multiplied by a nonlinear ``coefficient'' $F_3$
depending only on the first order increments. Expression \eqref{Phi12} will be most useful to deduce 
all properties of $\Phi_1$.

Instead, $\Phi_2$, that we study later on in this section, is a nonlinear and nonlocal zero order 
operator ---recall that is given by \eqref{Phi220}. For instance, we will see that $\Phi_2$ sends
the space of Lipschitz functions into itself.

We start studying $\Phi_1$, given by \eqref{Phi10}-\eqref{F1},
\begin{equation}\label{Phi1}
 \Phi_1 (a,\varphi):=\int_\R F_1(a,\delta_- \varphi)\ d\mu(t),
\end{equation}
where
$$
F_1(a,q):=\int_0^q \frac{d\tau}{(1+a^2\tau^2)^{\frac{2+\alpha}{2}}}.
$$

\begin{lemma}\label{lemPhi1}
 The operator $\Phi_1=\Phi_1 (a,\varphi): 
 \R\times X\to Y$ is of class $C^1$. 
\end{lemma}

\begin{proof}
Changing $t$ by $-t$ in \eqref{Phi1} and using that $\delta_-\va(s,-t)=\delta_+\va(s,t)$
(recall \eqref{deltaplus} for the definition of $\delta_+$), we have
$$
\Phi_1 (a,\varphi)=\frac{1}{2}\int_\R \left\{ F_1(a,\delta_- \varphi)+F_1(a,\delta_+ \varphi)
\right\} d\mu(t).
$$
Denoting
$$
q=\delta_- \va \quad\text{and}\quad p=\delta_+\va ,
$$
we have
\begin{eqnarray*}
F_1(a,q)+F_1(a,p)&=&\left\{\int_0^q +\int_0^p \right\} \frac{d\tau}{(1+a^2\tau^2)^{\frac{2+\alpha}{2}}}
=
\left\{\int_0^q +\int_{-p}^0 \right\} \frac{d\tau}{(1+a^2\tau^2)^{\frac{2+\alpha}{2}}}\\
&=& (q+p)\ \frac{1}{q+p}\int_{-p}^q \frac{d\tau}{(1+a^2\tau^2)^{\frac{2+\alpha}{2}}}.
\end{eqnarray*}
Making the change of variables $ \tau=-p+(q+p)\overline\tau$, we have that
$F_1(a,q)+F_1(a,p)=(q+p)F_3(a,q,p)$, where
$$
F_3(a,q,p):=\int_0^1 \frac{d\overline\tau}{\left\{ 1+ a^2\left( -p+(q+p)\overline\tau \right)^2 
\right\}^{\frac{2+\alpha}{2}}}.
$$
Therefore, we conclude that
\begin{equation}\label{Phi12}
\Phi_1 (a,\varphi):=\frac{1}{2}\int_\R (\delta_-\va+\delta_+\va)\ F_3(a,\delta_- \va, \delta_+\va)
\ d\mu(t).
\end{equation}
This is the expression that will be useful to establish the lemma.

We first collect several important inequalities for functions in $X$. For all $\va\in X$ and real numbers
$s,\overline s$, and $t$, we have
\begin{equation}\label{indel1}
 |\delta_-\varphi(s,t)|+|\delta_+\varphi(s,t)|\leq 2 \Vert\va\Vert_X,
\end{equation}
\begin{equation}\label{indel2}
 |(\delta_- \va +\delta_+\varphi)(s,t)|\leq 2 \Vert\va\Vert_X \, |t|^\beta,
\end{equation}
and
\begin{equation}\label{indel3}
 |\delta_\mp \varphi(s,t) -\delta_\mp \varphi(\overline s,t)|\leq \Vert\va\Vert_X \, |s-\overline s|\,
 |t|^{\beta -1}.
\end{equation}
Inequalities \eqref{indel2} and \eqref{indel3} are the crucial point to prove that $\Phi_1$ and its derivatives
send functions in $X$ to functions in $Y$. This will be accomplished with the inequalities following
\eqref{fracholder} ---an argument in H\"older spaces
already used in \cite{Sil} for the pure fractional Laplacian.

While \eqref{indel1} is obvious, \eqref{indel2} is easily proved as follows:
\begin{eqnarray*}
 |\va(s)-\va(s-t)+\va(s)-\va(s+t)|&=& \left| -\int_ 0^1 \frac{d}{d\rho} \va(s-\rho t)\, d\rho
 - \int_ 0^1 \frac{d}{d\rho} \va(s+\rho t) d\rho\right| \\
 &=& \left| \int_ 0^1 \left\{ \va'(s-\rho t) -\va'(s+\rho t) \right\} t\, d\rho\right| \\
 &\leq & \int_ 0^1 \Vert\va\Vert_X\, |2\rho t|^\beta |t|\, d\rho \leq 
 2 \Vert\va\Vert_X \, |t|^{\beta +1}.
\end{eqnarray*}
Similarly, we can establish \eqref{indel3}, as follows:
\begin{eqnarray*}
& & \hspace{-1.5cm}  \left|\frac{\va(s)-\va(s-t)}{|t|}-\frac{\va(\overline s)-\va(\overline s-t)}{|t|}\right| \\
&=& \left| \int_ 0^1 \frac{d}{d\rho} \, \frac{\va\left(\rho s+(1-\rho)\overline s\right) 
 -\va\left(\rho s+(1-\rho)\overline s-t\right)}{|t|}\, d\rho \right| \\
 &=& \left| (s-\overline s) \int_ 0^1  \frac{\va'\left(\rho s+(1-\rho)\overline s\right) -
 \va'\left(\rho s+(1-\rho)\overline s-t\right)}{|t|}\,  d\rho \right| \\
 & \leq & \Vert\va\Vert_X \,  |s-\overline s|\, |t|^{\beta -1}.
\end{eqnarray*}
The same bound for $\delta_+\va$ follows from $\delta_+\va(s,t)=\delta_-\va(s,-t)$.

We can now check that the integrand in \eqref{Phi12} is an integrable function, and thus
$\Phi_1(a,\va)$ is well defined by \eqref{Phi12}. That the function is integrable in $\{|t|\geq 1\}$
follows from \eqref{indel1} and $0\leq F_3\leq 1$ ---recall that $d\mu(t)=|t|^{-1-\alpha}\, dt$.
Instead for the integral in $(-1,1)$, we use  \eqref{indel2}, $0\leq F_3\leq 1$, and that
$\int_{-1}^{1}|t|^{\beta -1-\alpha}\, dt < \infty$.

We now verify that $\Phi_1(a,\va)\in Y$ for all $\va\in X$. We add and subtract a term to have 
$$
2\left\{ \Phi_1(a,\va)(s)- \Phi_1(a,\va)(\overline s)\right\}=\int_\R i_1(s,t)\, d\mu(t) +
\int_\R i_2(s,t)\, d\mu(t),
$$
where 
$$
i_1(s,t)=\left\{ (\delta_-\va+\delta_+\va)(s,t)-(\delta_-\va+\delta_+\va)(\overline s,t)\right\}
F_3\left(a,\delta_-\va(s,t),\delta_+\va(s,t)\right)
$$
and
$$
i_2(s,t)=(\delta_-\va+\delta_+\va)(\overline s,t) \left\{ F_3\left(a,\delta_-\va(s,t),\delta_+\va(s,t)\right)
- F_3\left(a,\delta_-\va(\overline s,t),\delta_+\va(\overline s,t)\right)\right\}.
$$

We claim that
\begin{equation}\label{i12}
 (|i_1|+|i_2|)(s,t)\leq C \min(|s-\overline s|\, |t|^{\beta-1},|t|^\beta ), 
\end{equation}
where $C$ is a constant depending only on an upper bound for $|a|+\Vert\va\Vert_X$.
Indeed, for $i_1$, this follows from \eqref{indel2}, \eqref{indel3}, and $0\leq F_3\leq 1$.
For $i_2$, the bound by $|t|^\beta$ follows from \eqref{indel2} and $0\leq F_3\leq 1$.
The first bound for $|i_2|$ follows from \eqref{indel1} and the use of the intermediate value formula to express 
the difference of values for $F_3$ in the second factor in $i_2$. Here we also use that
$(|\partial_qF_3|+|\partial_pF_3|)(a,q,p)\leq C a^2(|q|+|p|)\leq C$ ---the last inequality by \eqref{indel1}.

Using \eqref{i12}, we can prove that $\Phi_1(a,\va)\in Y$. Indeed, 
an $L^\infty$ bound for $\Phi_1(a,\va)$ has already been discussed above. Now
\begin{eqnarray}\label{fracholder}
 \left|  \Phi_1(a,\va)(s)- \Phi_1(a,\va)(\overline s)\right| &\leq & 
 C\int_\R (|i_1|+|i_2|)\, d\mu(t) \\
 &\leq&  C\int_{-|s-\overline s|}^{|s-\overline s|} |t|^\beta \frac{dt}{|t|^{1+\alpha}}
 +C\int_{\{|t|\geq|s-\overline s|\}} |s-\overline s| |t|^{\beta-1}\frac{dt}{|t|^{1+\alpha}} \nonumber \\
 &=& C|s-\overline s|^{\beta -\alpha} + C |s-\overline s|\, |s-\overline s|^{\beta -\alpha-1}
 =  C|s-\overline s|^{\beta -\alpha}. \nonumber
\end{eqnarray}

Next we check that $\Phi_1$ is differentiable with respect to $a$ and that $D_a \Phi_1$
is continuous from $\R\times X$ to $Y$. From \eqref{Phi12} we have that
\begin{equation}\label{daPhi1}
D_a \Phi_1(a,\va)=\frac{1}{2}\int_\R (\delta_-\va +\delta_+\va)\,  
\partial_aF_3 (a, \delta_-\va,\delta_+\va)\, d\mu(t). 
\end{equation}
This expression has exactly the same form as the one for $\Phi_1(a,\va)$, but with $F_3$ replaced
by $\partial_aF_3$. As for $F_3$, note that  $\partial_aF_3$ is a smooth and bounded function, since
$|q|+|p|=|\delta_-\va|+|\delta_+\va|\leq C$ by \eqref{indel1}. Thus, as before, $D_a \Phi_1$
is well defined and belongs to $Y$. 

To prove that $D_a \Phi_1$ is continuous from $\R\times X$ to $Y$, we take a sequence $(a_{k},\va_{k})$ converging
in $\R\times X$ to $(a,\va)$. We need to bound
$$
(D_a\Phi_1(a_{k},\va_{k}))(s)-(D_a\Phi_1(a,\va))(s)-(D_a\Phi_1(a_{k},\va_{k}))(\overline s)
+(D_a\Phi_1(a,\va))(\overline s).
$$
To do this, we need to add and subtract a good number of terms, and thus it is convenient to write the above
expression in a more compact (or ``symbolic'') way.
We write the integrand in \eqref{daPhi1} as the expression $G\phi (s) H\phi(s)$,
where $\phi=(a,\va)$, $\phi_k = (a_k,\varphi_k)$, $t$ is given,
$$
G\phi (s):=(\delta_-\va +\delta_+\va)(s,t),
$$
and
$$
H\phi(s):=\partial_aF_3 (a, \delta_-\va,\delta_+\va)(s,t).
$$

We must bound
$$
G\phi (s) H\phi(s)-G\phi_{k} (s) H\phi_{k}(s)-G\phi (\overline s) H\phi(\overline s)+G\phi_{k} (\overline s) H\phi_{k}(\overline s).
$$
We write this expression as
\begin{eqnarray*}
&& G\phi (s) (H\phi-H\phi_{k})(s) \\
&&+(G\phi-G\phi_{k}) (s) H\phi_{k}(s)\\
&& -G\phi (\overline s) (H\phi-H\phi_{k})(\overline s) \\
&& -(G\phi-G\phi_{k}) (\overline s) H\phi_{k}(\overline s).
\end{eqnarray*}
We add and subtract one term in the first and third lines (and another one in the second and fourth lines) 
to have this equal to
\begin{eqnarray*}
&& (G\phi (s)-G\phi(\overline s)) (H\phi-H\phi_{k})(s) \\
&&+(G\phi-G\phi_{k}) (s) (H\phi_{k}(s) - H\phi_{k}(\overline s))\\
&& +G\phi (\overline s) ((H\phi-H\phi_{k})(s) - (H\phi-H\phi_{k})(\overline s)) \\
&& +((G\phi-G\phi_{k}) (s)-(G\phi-G\phi_{k}) (\overline s)) H\phi_{k}(\overline s).
\end{eqnarray*}
Now, every of these four lines can be controlled in absolute value by the bound
$$
c(k) \min(|s-\overline s|\, |t|^{\beta-1},|t|^\beta ),
$$
where the constant $c(k)\to 0$ as $k\to\infty$. 
This is done in the way explained before, right after \eqref{i12}, taking care now to control the smallness of 
terms like
$$
|(H\phi-H\phi_{k})(s)|= |\partial_aF_3 (a, \delta_-\va,\delta_+\va)(s,t)-\partial_aF_3 (a_{k}, \delta_-\va_{k},\delta_+\va_{k})(s,t)|
$$
with the intermediate value theorem. Finally,
we can integrate in $t$ and proceed as in \eqref{fracholder} (and the
inequalities following it), to deduce that $D_a\Phi_1(a_{k},\va_{k})$ converges to $D_a\Phi_1(a,\va)$ in $Y$.

Finally, we need to prove the $C^1$ character of $\Phi_1$ with respect to the $\va$ variable.
We have that
\begin{equation}\label{diff}
\begin{split}
&  \hspace{-.5cm} 2D_\va\Phi_1(a,\va)\,\psi =\int_\R (\delta_-\psi +\delta_+\psi)\,  
F_3 (a, \delta_-\va,\delta_+\va)\, d\mu(t) \\
\ \hspace{.5cm} &+  \int_\R (\delta_-\va +\delta_+\va)\left\{  
\partial_qF_3 (a, \delta_-\va,\delta_+\va)\delta_-\psi + 
\partial_pF_3 (a, \delta_-\va,\delta_+\va)\delta_+\psi \right\} d\mu(t). 
\end{split}
\end{equation}
Recall that $q=\delta_-\va$, $p=\delta_+\va$, $\delta_-\psi$, and $\delta_+\psi$ are all bounded,
by \eqref{indel1}. Note that $\partial_qF_3$, $\partial_pF_3$, $\partial_{qq}F_3$,
$\partial_{qp}F_3$, and $\partial_{pp}F_3$ are all smooth and bounded functions. Now, to control
$| (D_\va\Phi_1(a,\va)\,\psi)(s)-(D_\va\Phi_1(a,\va)\,\psi)(\overline s)|$, we proceed
as we did above for $|\Phi_1(a,\va)(s)-\Phi_1(a,\va)(\overline s)|$ adding and subtracting one 
auxiliary term. In this way we see that $ D_\va\Phi_1(a,\va)\,\psi \in Y$.

The continuity of $D_\va\Phi_1$ as a function of $(a,\va)$ 
with values in the space of bounded linear operators from $X$ to $Y$ is proved in a similar way to 
the one above for the continuity of $D_a\Phi_1$. Here we must look at
$$
(D_\va\Phi_1(a_{k},\va_{k})\,\psi)(s)-(D_\va\Phi_1(a,\va)\,\psi)(s)-(D_\va\Phi_1(a_{k},\va_{k})\,\psi)(\overline s)
+(D_\va\Phi_1(a,\va)\,\psi)(\overline s),
$$
use expression \eqref{diff}, and add and subtract terms as above for $D_a\Phi_1$. The first integral \eqref{diff}
is easier to deal with, while the second has the same structure as in $D_a\Phi_1$.
\end{proof}

Next, we deal with the second term $\Phi_2$ in the operator $\Phi$. Recall
that it is given by expression \eqref{Phi220}-\eqref{F30},
\begin{equation}\label{Phi22}
\Phi_2(a,\lambda,\va)=\int_\R \delta_0\va \ F_2(t,a,\lambda,\delta_0\va)\, dt,
\end{equation}
where
\begin{equation}\label{F3}
 F_2(t,a,\lambda,r)= \int_0^1 \frac{d\overline\tau}{\left\{ t^2+(2\lambda R+ ar
\overline\tau)^2\right\}^{\frac{2+\alpha}{2}}}.
\end{equation}

With this expression at hand, we prove our last lemma.

\begin{lemma}\label{lemPhi2}
 There exists $\nu>0$ small enough $($depending only on $\alpha)$ for which the operator 
 $\Phi_2=\Phi_2 (a,\lambda,\varphi): (-\nu,\nu)\times (1/2,3/2)\times B_{10}(0)\subset \R\times\R\times
 X\to Y$ is of class~$C^1$. 
\end{lemma}

\begin{proof}
Recall that $R>0$ has been chosen to depend only on $\alpha$. 
We take $|a|\leq \nu$, with $\nu$ small depending on how large
is $R=R(\alpha)$.
Since $|r|=|\delta_0\va|\leq 2\Vert\va\Vert_\infty \leq 20$,
$\lambda\in (1/2,3/2)$, and $\overline\tau\in [0,1]$, 
in the definition \eqref{F3} of $F_2$ we have
$$
R/2\leq R-20a\leq 2\lambda R + ar\overline\tau\leq 3R + 20 a\leq 4R.
$$
It follows that $F_2$ is a smooth, positive, and bounded function of $(t,a,\lambda,r)
\in \R\times (-\nu,\nu)\times(1/2,3/2)\times(-20,20)$. In addition, we have that
\begin{equation}\label{decay}
 \lim_{|t|\to\infty}\frac{F_2(t,a,\lambda,r)}{|t|^{-2-\alpha}} = 1
\end{equation}
uniformly as $(a,\lambda,r)$ belong to the above sets.

It follows that the integrand in \eqref{Phi22} is integrable, and hence $\Phi_2$ is well defined.

Next, we see that $\Phi_2(a,\lambda,\va)\in Y$ and has $Y$-norm controlled by an upper bound on
$|a|+\lambda+\Vert\va\Vert_X$. Indeed, we can see even more: that $\Phi_2(a,\lambda,\va)$ 
is Lipschitz if $\va$ is Lipschitz, and $\Phi_2(a,\lambda,\va)$ 
has Lipschitz norm controlled by an upper bound on
$|a|+\lambda+\Vert\va\Vert_{\text{Lip}}$. To verify this, the $L^\infty$ norm of $ \Phi_2(a,\lambda,\va)$
has already been treated above. Next, we look at $\Phi_2(a,\lambda,\va)(s)-
\Phi_2(a,\lambda,\va)(\overline s)$, add and subtract the term $\int_\R \delta_0\va(\overline s,t)
\, F_2(t,a,\lambda,\delta_0\va(s,t))\, dt$, use that $F_2$ and $\partial_rF_2$ are bounded, 
use the intermediate value formula for the last two terms, and that
$$
|\delta_0\va(s,t)-\delta_0\va (\overline s,t)|\leq 2\Vert\va\Vert_{\text{Lip}} |s-\overline s|
$$
to conclude the desired bound. For the convergence of the integrals, use \eqref{decay} and that $|\partial_r F_2|\leq C 
|t|^{-3-\alpha}$ for $|t|$ large.

Finally, we must prove that $\Phi_2$ is $C^1$ with respect to $(a,\lambda,\va)$ with values in $Y$.
For this, we compute $D_{a}\Phi_{2}$, $D_{\lambda}\Phi_{2}$, and $D_{\va}\Phi_{2}$ at $(a,\lambda,\va)$ as we
did for $\Phi_{1}$ in the proof of Lemma~\ref{lemPhi1}. The same argument as in the previous paragraph gives that
these functions are Lipschitz and have their Lipschitz norm controlled by an upper bound on
$|a|+\lambda+\Vert\va\Vert_{\text{Lip}}$. From this, it follows the continuity of  $D_{a}\Phi_{2}$, $D_{\lambda}\Phi_{2}$, and $D_{\va}\Phi_{2}$ 
as  functions of $(a,\lambda,\va)$ with values in~$Y$.
Indeed, if $(a_k,\lambda_{k},\va_k)\to (a,\lambda,\va)$ in $\R\times \R\times X$ then, by the previous bounds, we have that
the sequence $D_{(a,\lambda,\va)} \Phi_2(a_k,\lambda_{k},\va_k)$ is uniformly bounded in the Lipschitz norm. 
By Arzel\`a-Ascoli theorem, a subsequence converges in the weaker H\"older norm of $Y$ to a function in $Y$. 
This function must be $D_{(a,\lambda,\va)} \Phi_2(a,\lambda,\va)$,
since the integrals in the expressions for $D_{(a,\lambda,\va)} \Phi_2(a_k,\lambda_{k},\va_k)$ convergence to 
the corresponding integrals for $(a,\lambda,\va)$. This follows from the dominated convergence theorem. 
As a consequence, the full sequence converges in $Y$ to $D_{(a,\lambda,\va)} \Phi_2(a,\lambda,\va)$.
\end{proof}

\end{document}